\documentclass[12pt]{amsart}
\usepackage{latexsym}
\usepackage{amsmath,amssymb,amsfonts}
\usepackage{mathrsfs}
\usepackage[abbrev]{amsrefs}
\usepackage[dvipdfmx]{graphicx}

\usepackage{latexsym}
\usepackage{delarray} 

\newtheorem{thm}{Theorem}[section]

\newtheorem{prop}[thm]{Proposition}
\newtheorem{lem}[thm]{Lemma}

\newcommand{\Rn}{\mathbb{R}^n}

\newcommand{\ep}{\varepsilon}

\newcommand\va{\phi}
\newcommand\R{\mathbb{R}}

\newcommand{\mathH}{\mathcal{H}}

\numberwithin{equation}{section}
\numberwithin{thm}{section}

\makeatletter
\newcommand{\vertiii}[1]{{\left\vert\kern-0.25ex\left\vert\kern-0.25ex\left\vert #1 
    \right\vert\kern-0.25ex\right\vert\kern-0.25ex\right\vert}}

\title[Bilinear estimates in Besov spaces]{
Bilinear estimates in Besov spaces generated by the Dirichlet Laplacian}
\author[T. Iwabuchi, T. Matsuyama and 
K. Taniguchi]{Tsukasa Iwabuchi, Tokio Matsuyama 
and Koichi Taniguchi}
\address{ 
Tsukasa Iwabuchi \endgraf 
Mathematical Institute \endgraf
Tohoku University \endgraf 
Aoba \endgraf 
Sendai 980-8578 \endgraf
Japan  
}
\email{t-iwabuchi@m.tohoku.ac.jp}
\address{ 
Tokio Matsuyama \endgraf 
Department of Mathematics \endgraf 
Chuo University \endgraf 
1-13-27, Kasuga, Bunkyo-ku \endgraf 
Tokyo 112-8551 \endgraf 
Japan}
\email{tokio@math.chuo-u.ac.jp} 
\address{ 
Koichi Taniguchi \endgraf 
Department of Mathematics \endgraf 
Chuo University \endgraf 
1-13-27, Kasuga, Bunkyo-ku \endgraf 
Tokyo 112-8551 \endgraf 
Japan} 
\email{koichi-t@gug.math.chuo-u.ac.jp} 
\thanks{
 The first author was supported by 
 Grant-in-Aid for Young
 Scientists Research (A) (No. 17H04824), 
 Japan Society for the Promotion of Science.
 The second author was supported by 
Grant-in-Aid for Scientific 
Research (C) (No. 18K03377), 
Japan Society for the Promotion of Science. 
}


\begin{document}



\begin{abstract}
The purpose of this paper is to establish 
bilinear estimates in Besov spaces 
generated by the Dirichlet Laplacian on 
a domain of Euclidian spaces. 
These estimates are proved 
by using the gradient estimates for heat semigroup 
together with the Bony paraproduct formula and  
the boundedness of spectral multipliers. 
\end{abstract}

\maketitle


\section{Introduction}
The bilinear estimates in Sobolev spaces or Besov spaces 
are of great importance 
to study the well-posedness for the 
Cauchy problem to nonlinear partial 
differential equations. 
In this paper we study 
the bilinear estimates in Besov spaces:
\begin{equation}\label{EQ:C-biliear}
\|fg\|_{\dot{B}^{s}_{p,q}} 
\le 
C\left(
\|f\|_{\dot{B}^{s}_{p_1,q}}
\|g\|_{L^{p_2}}
+
\|f\|_{L^{p_3}}
\|g\|_{\dot{B}^{s}_{p_4,q}}
\right),
\end{equation}
where $s>0$ and 
$p$, $p_1,p_2,p_3,p_4$ and 
$q$ satisfy
$$
1\le p, p_1,p_2,p_3,p_4, q\le \infty 
\quad \text{and}
\quad \frac{1}{p}=\frac{1}{p_1}+\frac{1}{p_2}=\frac{1}{p_3}+\frac{1}{p_4}.
$$
We study also the inhomogeneous version 
of \eqref{EQ:C-biliear}. 
\\

The basis of proving the bilinear estimates in Sobolev spaces 
$W^{k,p}$ $(k = 1,2,\ldots)$ is to use the Leibniz rule and the H\"older inequality. 
However, when one considers
 the fractional order regularity,  
some idea would be needed. 
If the domain is the whole space $\mathbb R^n$, 
the Fourier transformation is one of the most powerful tools,   
and allows one to introduce the derivative of fractional order. 
It enables us to prove the bilinear estimates 
by using frequency decomposition called the Bony paraproduct 
formula (see Bony \cite{Bo-1981}) and the boundedness of Fourier multipliers. 
On the other hand, when the domain is 
different from $\mathbb R^n$, 
one cannot rely on such a kind of method.
It will be revealed that the bilinear estimates hold for 
small regularity number in the Besov spaces generated by the 
Dirichlet Laplacian,  
of which we established several properties 
on open sets in $\R^n$ (see \cite{IMT-Besov}),
and that there arises a problem for large regularity 
essentially.
The purpose of this paper is to establish the bilinear estimates 
in those Besov spaces.  \\

In the rest of this section we give a 
definition of Besov spaces generated by the Dirichlet Laplacian 
on an open set along \cite{IMT-Besov}. 
Let $\Omega$ be an open set of $\mathbb R^n$ with $n\ge1$. 
We denote by $\mathH$ the self-adjoint realization of the Dirichlet Laplacian $-\Delta$ with the domain 
\begin{equation*}\label{EQ:D-Laplacian}
\mathcal D (\mathH) = \left\{ f \in H^1_0 (\Omega) \,\big|\, \mathH f \in L^2(\Omega) \right\} 
\end{equation*}
such that 
\[
\left(\mathH f,g\right)_{L^2(\Omega)}
=\int_\Omega \nabla f(x)\cdot \overline{\nabla g(x)}\, dx 
\]
for any $f \in \mathcal D(\mathH)$ and $g\in H^1_0(\Omega)$, 
where $(\cdot,\cdot)_{L^2(\Omega)}$ stands for the inner product of $L^2(\Omega)$, and 
$H^1_0 (\Omega)$ is the completion of $C^\infty _0 (\Omega)$ 
with respect to $H^1(\Omega)$-norm. 
The operator $\mathH$ is a non-negative self-adjoint operator on $L^2(\Omega)$. 
For a Borel measurable function $\phi$ on $\mathbb R$, 
an operator $\phi(\mathH)$ is defined by letting
\[
\phi(\mathH) = \int^{\infty}_{0}%
\phi (\lambda) \, d E_{\mathH}(\lambda)  
\]
with the domain 
\[
\mathcal D ( \phi(\mathH)) =\left\{ f \in L^2 (\Omega) \, \bigg| \, 
\int^\infty_{0}
|\phi (\lambda)|^2 d \| E_{\mathH} (\lambda) f \|_{L^2 (\Omega)} ^2 < \infty
\right\},
\]
where $\{ E_{\mathH}(\lambda) \}_{\lambda \in \mathbb{R}}$ is 
the spectral resolution of the identity for $\mathH$. 

\vspace{0.5cm}

We begin by introducing the spaces of test functions on $\Omega$ and 
their duals,
which provide the basis for the study of our Besov spaces.
For this purpose, let us introduce 
the Littlewood-Paley 
partition of unity.  
Let $\phi_0$  
be a non-negative and smooth function on $\mathbb R$ such that  
\begin{equation*}
\label{EQ:phi1}
{\rm supp \, } \phi _0
\subset \{ \, \lambda \in \mathbb R \,\big|\,
2^{-1} \leq \lambda \leq 2 \, \}
\quad \text{and}\quad
\sum _{ j =-\infty}^\infty \phi_0 ( 2^{-j}\lambda) 
 = 1 
 \quad \text{for } \lambda > 0,  
\end{equation*}
and $ \{ \phi_j \}_{j=-\infty}^\infty$ is 
defined by letting 
\begin{equation*} \label{EQ:phi2}
\phi_j (\lambda) := \phi_0 (2^{-j} \lambda) 
 \quad \text{for }  \lambda \in \mathbb R . 
\end{equation*}

\vspace{0.5cm}

\noindent 
{\bf Definition (Spaces of test functions and 
distributions on $\Omega$).} 
\begin{enumerate}
\item[(i)] {\rm (}Linear topological spaces
$\mathcal X (\Omega)$ and $\mathcal X^\prime (\Omega)${\rm ).}
A linear topological space 
$\mathcal X (\Omega)$ is 
defined by letting
\begin{equation}\notag 
\mathcal X (\Omega) 
:= \left\{ f \in  L^1 (\Omega) \cap \mathcal D (\mathH) 
 \, \Big| \, 
    \mathH^{M} f \in L^1(\Omega ) \cap \mathcal D (\mathH) \text{ for any } M \in \mathbb N 
   \right\} 
\end{equation} 
equipped with the family of semi-norms $\{ p_{M}
 (\cdot) \}_{ M = 1 } ^\infty$ 
given by 
\begin{equation*}
\label{EQ:p_M}
p_{M}(f) := 
\| f \|_{ L^1(\Omega)} 
+ \sup _{j \in \mathbb N} 2^{Mj} 
  \| \phi_j (\sqrt{\mathH}) f \|_{ L^1(\Omega)} . 
\end{equation*}
$\mathcal X'(\Omega)$ denotes the topological dual of 
$\mathcal X (\Omega)$.

\item[(ii)] {\rm (}Linear topological spaces
$\mathcal Z (\Omega)$ and $\mathcal Z^\prime (\Omega)${\rm ).} 
A linear topological space 
$\mathcal Z (\Omega)$ is 
defined by letting
\begin{equation}\notag 
\mathcal Z (\Omega) 
:= \left\{ f \in \mathcal X (\Omega) 
 \, \Big| \, 
  \sup_{j \leq 0} 2^{ M |j|} 
    \big\| \phi_j \big(\sqrt{\mathH} \big ) f \big\|_{L^1(\Omega)} < \infty 
  \text{ for any } M \in \mathbb N
   \right\} 
\end{equation}
equipped 
with the family of semi-norms $\{ q_{M} (\cdot) \}_{ M = 1}^\infty$ 
given by 
\begin{equation*}\label{EQ:qV} 
q_{M}(f) := 
\| f \|_{L^1 (\Omega) }
+ \sup_{j \in \mathbb Z} 2^{M|j|} \| \phi_j (\sqrt{\mathH}) f \|_{L^1(\Omega)}. 
\end{equation*}
$\mathcal Z'(\Omega)$ denotes the topological dual of $\mathcal Z (\Omega)$.
\end{enumerate}

\vspace{5mm}

In this paper we often use the notation 
${}_{X'}\langle \cdot, \cdot \rangle_{X}$
of duality pair of a linear topological space 
$X$ and its dual $X'$.\\

Let us give a few remarks on the spaces $\mathcal X (\Omega)$, $\mathcal Z (\Omega)$ and their dual spaces.
The spaces $\mathcal X (\Omega)$ and $\mathcal Z (\Omega)$ are non-empty, since
\[
\phi(\mathH)f \in \mathcal Z(\Omega) \subset \mathcal X(\Omega)\quad \text{for any $f \in L^1(\Omega)\cap L^2(\Omega)$ and $\phi \in C^\infty_0((0,\infty))$}
\]
by Lemma \ref{lem:Lp} below.
It is proved in Lemma 4.2 from \cite{IMT-Besov} that $\mathcal X (\Omega)$ and $\mathcal Z (\Omega)$ are complete, 
and
in Lemma 4.6 from \cite{IMT-Besov} 
that 
\begin{equation}
\label{EQ:inc1}
\mathcal X(\Omega) \hookrightarrow L^p(\Omega) \hookrightarrow \mathcal X'(\Omega),
\end{equation}
\begin{equation}
\label{EQ:inc2}
\mathcal Z(\Omega) \hookrightarrow L^p(\Omega) \hookrightarrow \mathcal Z'(\Omega)
\end{equation}
for any $1\le p\le \infty$. 
The inclusion relation \eqref{EQ:inc1} (\eqref{EQ:inc2} resp.) assures that 
$$
\int_ {\Omega} \big|f (x) \overline{g (x) } \big| \, dx  < \infty  
$$
for any $f \in L^p (\Omega)$, $1 \leq p \leq \infty$, and $g \in \mathcal X (\Omega)$ 
($g \in \mathcal Z (\Omega) $ resp.). 
Hence we can regard 
functions in the Lebesgue spaces 
as elements in $\mathcal X'(\Omega)$ and $\mathcal Z'(\Omega)$ as follows:\\

\noindent 
{\bf Definition.}
For $f\in L^1(\Omega)+L^\infty(\Omega)$,  
we identify 
$f$ as an element in $\mathcal{X}^\prime(\Omega)$
{\rm (}$\mathcal{Z}^\prime(\Omega)$ resp.{\rm )}
by letting
\begin{equation*}
\label{EQ:regard}
{}_{\mathcal{X}'(\Omega)} \langle f , g \rangle _{\mathcal{X}(\Omega)} 
=  \int_\Omega f(x)\overline{g(x)}\, dx 
\quad \left(
{}_{\mathcal{Z}'(\Omega)} \langle f , g \rangle _{\mathcal{Z}(\Omega)} 
=  \int_\Omega f(x)\overline{g(x)}\, dx
\quad \mathrm{resp.}\right)
\end{equation*}
for any $g\in \mathcal{X}(\Omega)$ 
{\rm (}$g\in \mathcal{Z}(\Omega)$ 
resp.{\rm )}. 

\vspace{0.5cm}

Next, we introduce the notion of dual operators
on $\mathcal{X}^\prime(\Omega)$ and 
$\mathcal{Z}^\prime(\Omega)$.\\

\noindent 
{\bf Definition (Dual operators).} 
Let $\phi$ be a real-valued 
Borel measurable function on $\mathbb R$. 
\begin{enumerate}
\item[(i)] 
For a mapping $\phi (\mathH): \mathcal X (\Omega) \to \mathcal X (\Omega)$, 
we define $\phi (\mathH): \mathcal X' (\Omega) \to \mathcal X' (\Omega)$ 
by letting 
\begin{equation}\label{EQ:X'}
{}_{{\mathcal X'(\Omega)} }
\big\langle \phi (\mathH) f , g \big\rangle _{\mathcal X (\Omega)} 
:= 
{}_{{\mathcal X'(\Omega)}}\big\langle f , \phi (\mathH) g \big\rangle_{\mathcal X (\Omega)}  
\end{equation}
for any $f\in \mathcal X^\prime(\Omega)$ and $g\in \mathcal X(\Omega)$. 

\item[(ii)] 
For a mapping $\phi (\mathH) : \mathcal Z(\Omega) \to \mathcal Z(\Omega)$, 
we define $\phi (\mathH) : \mathcal Z'(\Omega) \to \mathcal Z'(\Omega)$ 
by letting 
\begin{equation*}\label{EQ:Z'}
{}_{\mathcal{Z}'(\Omega)} \big\langle \phi (\mathH) f , g \big\rangle _{\mathcal{Z}(\Omega)}
:= 
{}_{\mathcal{Z}'(\Omega)} \big\langle f , \phi (\mathH) g \big\rangle _{\mathcal{Z}(\Omega)}
\end{equation*}
for any $f\in \mathcal Z^\prime(\Omega)$ and $g\in \mathcal Z(\Omega)$. 
\end{enumerate}

\vspace{5mm}

When we consider the inhomogeneous Besov spaces, 
a function $\psi$, whose support 
is restricted in the neighborhood 
of the origin,
is needed. 
More precisely, let $\psi \in 
C^\infty_0(\mathbb R)$ be a function
satisfying 
\begin{equation*} \label{EQ:psi}
\psi (\lambda^2) 
 + \sum _{ j=1}^\infty  \phi_j (\lambda) 
 = 1 
 \quad \text{for } \lambda \geq 0. 
\end{equation*}

\vspace{5mm}

We are now in a position to give the 
definition of Besov spaces
generated by $\mathcal{H}$.\\

\noindent 
{\bf Definition (Besov spaces).} 
Let $s \in \mathbb R$ and $1 \leq p,q \leq \infty$. Then the Besov spaces 
are defined as follows:  
\begin{enumerate}
\item[(i)] 
The inhomogeneous
Besov spaces $B^s_{p,q} (\mathH) $ are defined by letting 
\begin{equation}\notag
B^s_{p,q} (\mathH)
:= \left\{ 
    f \in \mathcal X'(\Omega) 
     \,\Big|\,
     \| f \|_{B^s_{p,q} (\mathH)} < \infty
   \right\} , 
\end{equation}
where
\begin{equation}\notag 
     \| f \|_{B^s_{p,q} (\mathH)} 
       := \| \psi (\mathH) f \|_{L^p (\Omega) } 
          + \left\| \big\{ 2^{sj} \| \phi_j (\sqrt{\mathH}) f \|_{L^p (\Omega) }  
                 \big\}_{j \in \mathbb  N}
          \right\|_{\ell^q (\mathbb N)}. 
\end{equation}

\item[(ii)] 
The homogeneous
Besov spaces $\dot B^s_{p,q} (\mathH) $ are defined by letting
\begin{equation} \notag 
\dot B^s_{p,q}  (\mathH)
:= \left\{ f \in \mathcal Z'(\Omega) 
     \,\Big|\, 
     \| f \|_{\dot B^s_{p,q}(\mathH)} < \infty 
   \right\} , 
\end{equation}
where
\begin{equation}\notag 
     \| f \|_{\dot B^s_{p,q}(\mathH)} 
       := \left\| \big\{ 2^{sj} \| \phi_j (\sqrt{\mathH}) f\|_{L^p(\Omega)}
                 \big\}_{j \in \mathbb Z}
          \right\|_{\ell^q (\mathbb Z)}. 
\end{equation}
\end{enumerate}

\vspace{0.5cm}

It is proved in Theorem 2.5 from 
\cite{IMT-Besov} that  
$B^s_{p,q}(\mathH)$ and $\dot B^s_{p,q}(\mathH)$ 
are Banach spaces, and 
\[
\mathcal{X}(\Omega)\hookrightarrow B^s_{p,q}(\mathH)
\hookrightarrow \mathcal{X}^\prime(\Omega),
\]
\[
\mathcal{Z}(\Omega)\hookrightarrow 
\dot{B}^s_{p,q}(\mathH)
\hookrightarrow \mathcal{Z}^\prime(\Omega).
\]
for any $s\in \mathbb R$ and $1\le p,q\le \infty$.

\vspace{5mm}


We conclude this section by giving two remarks; 
the first one is 
the regularity numbers such that the bilinear estimates hold, and 
the second is about necessity of the assumption on the gradient estimate \eqref{EQ:grad-infty}.
As is well known, when $\Omega$ is the whole 
space $\R^n$, one does not need to impose any 
restriction on the regularity number $s>0$ of 
Besov spaces. 
However, when we consider these estimates 
for functions whose regularity is measured by the Dirichlet 
Laplacian $\mathH$ on domains, 
a restriction is required on the regularity.
In fact, it is possible to construct a 
counter-example for high regularity 
(see appendix \ref{App:AppendixA}).
This is because $\mathH(fg)$ 
does not necessarily belong to 
$\mathcal D(\mathH)$ even if $f$ and $g$ belong to 
$\mathcal D(\mathH ^2)$.
This can be seen from the following observation: 
Let $\Omega$ be 
a domain with smooth boundary. 
Applying 
the Leibniz rule to $\mathcal H (fg)$, 
we are confronted with the term $\nabla f \cdot \nabla g$ 
which does not belong to $\mathcal D(\mathH)$,
since it does not in general vanish
on the boundary.  
Here, we refer to a paper \cite{T-2018} in 
which the one dimensional differential operator $\partial_x$ 
maps functions involved with the Dirichlet boundary condition 
into those with the Neumann one, and vice versa.
Hence, in general, it is impossible  to 
get the estimates in high regularity.

As to the second remark, as far as our proof of main theorem is concerned, we need to  
estimate the derivative of functions. Therefore, the gradient estimates 
for heat semigroup in $L^\infty$ or even $L^p$ are required.\\


This paper is organized as follows. 
In \S\ref{sec:sec2} we state the main result. 
In \S\ref{sec:sec3} 
we prepare some useful lemmas to prove the main theorem. In \S\ref{sec:sec4} 
we prove the main theorem. 
In \S\ref{sec:sec5} 
we discuss 
the bilinear estimates
in the spaces generated by the Schr\"odinger operators.

\section{Statement of result} \label{sec:sec2}

Let us consider a domain $\Omega$ such that  
the following gradient estimate
\begin{equation}\label{EQ:grad-infty}
\|\nabla e^{-t\mathH}\|_{L^\infty(\Omega)\to L^\infty(\Omega)} \le Ct^{-\frac{1}{2}}
\end{equation}
holds either for any $t \in(0,1]$ or for any $t>0$, 
where $\{e^{-t\mathH}\}_{t>0}$ is the semigroup generated by $\mathH$. 
\\

We shall prove here the following:

\begin{thm}\label{thm:bilinear}
Let $0<s<2$ and $p$, $p_1,p_2,p_3,p_4$ and 
$q$ be such that
$$
1\le p, p_1,p_2, p_3,p_4, q\le \infty 
\quad \text{and}
\quad \frac{1}{p}=\frac{1}{p_1}+\frac{1}{p_2}=\frac{1}{p_3}+\frac{1}{p_4}.
$$
Then the following assertions hold{\rm :}
\begin{itemize}
\item[(i)] 
Let $\Omega$ be a domain of $\mathbb R^n$ such that  
\eqref{EQ:grad-infty} holds for any $t \in(0,1]$. 
Then there exists a constant $C>0$ such that 
\begin{equation}\label{EQ:bilinear1}
\|fg\|_{B^s_{p,q}(\mathH)} 
\le 
C\left(
\|f\|_{B^s_{p_1,q}(\mathH)}
\|g\|_{L^{p_2}(\Omega)}
+
\|f\|_{L^{p_3}(\Omega)}
\|g\|_{B^s_{p_4,q}(\mathH)}
\right)
\end{equation}
for any 
$f\in B^s_{p_1, q}(\mathH)\cap L^{p_3}(\Omega)$ 
and $g\in B^s_{p_4, q}(\mathH) \cap L^{p_2}(\Omega)$. 
\item[(ii)]
Let $\Omega$ be a domain of $\mathbb R^n$ such that  
\eqref{EQ:grad-infty} holds for any $t>0$. 
Then there exists a constant $C>0$ such that 
\begin{equation}\label{EQ:bilinear2}
\|fg\|_{\dot{B}^s_{p,q}(\mathH)} 
\le 
C\left(
\|f\|_{\dot{B}^s_{p_1,q}(\mathH)}
\|g\|_{L^{p_2}(\Omega)}
+
\|f\|_{L^{p_3}(\Omega)}
\|g\|_{\dot{B}^s_{p_4,q}(\mathH)}
\right)
\end{equation}
for any 
$f\in \dot{B}^s_{p_1, q}(\mathH)\cap L^{p_3}(\Omega)$ and $g\in \dot{B}^s_{p_4, q}(\mathH)\cap 
L^{p_2}(\Omega)$. 
\end{itemize}
\end{thm}

\vspace{5mm}

As to the range of the regularity number $s$ in Theorem \ref{thm:bilinear}, 
it is not clear whether or not it is sharp. 
However, we can find an $s\ge2$ such that 
Theorem \ref{thm:bilinear} does not hold. 
This topic is discussed in appendix \ref{App:AppendixA}. \\

When $\Omega$ is the whole space $\mathbb R^n$ or the half space $\mathbb R^n_+$ with $n\ge1$, 
we observe from the explicit representation formula of the heat kernels that 
the estimate \eqref{EQ:grad-infty} holds for any 
$t>0$. 
In the rest of this section, 
we give examples of domains such that \eqref{EQ:grad-infty} holds, 
and other examples of domains
where the bilinear estimates still hold for $p$ in some 
restricted ranges. 



%


\begin{enumerate} 
\item[(i)] 
When $\Omega$ is a domain with uniform $C^{2,\alpha}$-boundary for some 
$\alpha\in(0,1)$, \eqref{EQ:grad-infty} holds for any $t\in(0,1]$ 
(see Fornaro, Metafune and Priola \cite{FMP-2004}). 
Hence, the bilinear estimate \eqref{EQ:bilinear1} 
in Theorem \ref{thm:bilinear} holds in such a domain. 
In particular, when $\Omega$ is bounded, 
\eqref{EQ:grad-infty} holds for 
any $t>0$, since the infimum of the spectrum is strictly 
positive (see, e.g., Taniguchi \cite{T-pre} and the references therein). 
Hence, the bilinear estimate \eqref{EQ:bilinear2} in Theorem \ref{thm:bilinear} holds.
\\

%
%
%

\item[(ii)] Let $\Omega$ be an open set in $\Rn$.
Then there exists an exponent $p_0=p_0(\Omega)  \in [2,\infty]$ 
depending on $\Omega$   
such that if $p \in [1,p_0]$, then 
\begin{equation}
\label{EQ:grad-p}
\|\nabla e^{-t\mathH}\|_{L^{p}(\Omega)\to L^{p}(\Omega)} \le Ct^{-\frac{1}{2}}, 
\quad t>0.
\end{equation}
Here we note that \eqref{EQ:grad-p}
was proved for $p \in [1,2]$ in \cite{IMT-bdd}. 
In this case, it should be mentioned that 
we can prove the estimates 
\eqref{EQ:bilinear1} and \eqref{EQ:bilinear2} 
for $1\le p,p_1,p_2,p_3,p_4 \le p_0$ 
by performing some trivial modifications of the proof of Theorem \ref{thm:bilinear}.
\end{enumerate}

\vspace{5mm}

Finally, let us 
mention some domains and the range of $p$ such that \eqref{EQ:grad-p} holds.
\\

\begin{itemize}
\item[(a)] 
Let $n \ge 3$. Assume that $\Omega$ is the exterior domain of a compact set with $C^{1,1}$-boundary. 
Then \eqref{EQ:grad-p} holds for any 
$p\in[1,n]$ (see Theorem 2.1 from Georgiev 
and Taniguchi \cite{GeoTan-pre}). 
In this case we may take $ p_0 = p_0(\Omega)=n$. 
\end{itemize}

\vspace{5mm}

We are able to take domains and $p$ 
such that the Riesz transform is bounded,  
namely, $L^p$-boundedness of $\nabla \mathcal H ^{-\frac{1}{2}}$ 
implies the gradient estimate: 
$$
\|\nabla e^{-t\mathH}f\|_{L^p(\Omega)}
= 
t^{-\frac{1}{2}}\|\nabla \mathcal H^{-\frac{1}{2}} 
 \mathcal (tH) ^{\frac{1}{2}}e^{-t\mathH}f\|_{L^p(\Omega)}
\leq C t^{-\frac{1}{2}} \| f \|_{L^p (\Omega)}
$$
for $t > 0$. 
Hence, the following results are 
immediate consequences of (a) with $p=1$ and $L^p$-boundedness of the Riesz transform for some $p = p_0$
in \cites{Dah-1979,JK-1995} (see also \cites{Shen-1995,Shen-2005,Wood-2007}). \\

\begin{itemize}
\item[(b)] 
Let $n\ge2$. If $\Omega$ is a bounded domain with $C^1$-boundary, 
then \eqref{EQ:grad-p} holds for any 
$p\in[1,\infty)$. In this case we 
may take $p_0$ as any finite number.  \\

\item[(c)] Let $n\ge2$. If $\Omega$ is a bounded and Lipschitz domain, then 
\eqref{EQ:grad-p} holds for any 
$p\in[1, p_0]$, where 
$p_0=3$ for $n\ge3$ and $p_0=4$ for $n = 2$. 
\end{itemize}

\section{Preliminaries} \label{sec:sec3}
In this section we introduce 
some useful lemmas to prove Theorem \ref{thm:bilinear}. 
Here and below, we denote 
by $\mathscr S(\mathbb R)$ the space of all 
rapidly decreasing functions on $\R$.

\subsection{Approximations of the identity} 
The following results can be found in 
our previous paper \cite{IMT-Besov}.
The first one is the following. 
\begin{lem}[Lemma 4.5 from \cite{IMT-Besov}]
\label{lem:decomposition1}
Let $\Omega$ be an open set of $\R^n$. Then the following assertions hold{\rm :}
\begin{enumerate}
\item[(i)] 
For any $f \in \mathcal X(\Omega)$, we have
\begin{equation}\label{907-1}
f = \psi (\mathcal{H}) f 
    + \sum_{j=1}^\infty \phi_j 
    (\sqrt{\mathcal{H}}) f 
   \quad \text{in }
    \mathcal X(\Omega). 
\end{equation}
Furthermore, for any $f \in \mathcal X' (\Omega)$, 
we have also  the identity \eqref{907-1} 
in $\mathcal X ' (\Omega)$, 
and $\psi (\mathcal{H}) f$ and $\phi_j 
(\sqrt{\mathcal{H}}) f$ are regarded as elements 
in $L^\infty (\Omega)$. 

\item[(ii)] For any $f \in \mathcal Z(\Omega)$, we have 
\begin{equation}\label{907-2}
f =  \sum _{ j=-\infty}^\infty \phi_j 
(\sqrt{\mathcal{H}}) f 
   \quad \text{in }
    \mathcal Z (\Omega). 
\end{equation}
Furthermore, for $f \in \mathcal Z'(\Omega)$, we have also the identity \eqref{907-2} 
in $\mathcal Z'(\Omega)$, 
and $\phi_j (\sqrt{\mathcal{H}}) f$ are regarded as  elements in $L^\infty (\Omega)$.  
\end{enumerate}
\end{lem}

The second one is the following.
\begin{lem}\label{lem:lem-L2}
Let $\Omega$ be an open set of $\R^n$. Then 
the following assertions hold{\rm :} 
\begin{itemize}
\item[(i)]
For any $f \in L^2(\Omega)$ and $j\in \mathbb Z$,
we have 
\begin{equation*}
\label{EQ:identity1}
f= \psi(2^{-2j}\mathH)f
+
\sum^\infty_{k=j+1}\phi_k(\sqrt{\mathH})f \quad 
\text{in }L^2(\Omega)
\end{equation*}
and 
\begin{equation*}
\label{EQ:identity2}
f= 
\sum^j_{k=-\infty}\phi_k(\sqrt{\mathH})f 
+
\sum^\infty_{k=j+1}\phi_k(\sqrt{\mathH})f \quad 
\text{in }L^2(\Omega)
\end{equation*}
\item[(ii)] Let $1 \le p < \infty$. Then for any $f \in L^p(\Omega)$,
we have 
\begin{equation}\label{EQ:p-converge}
f =  \sum _{ j=-\infty}^\infty \phi_j 
(\sqrt{\mathcal{H}}) f 
   \quad \text{in }
    \mathcal X' (\Omega). 
\end{equation}
\end{itemize}
\end{lem}
\begin{proof}
The assertion (i) is proved in the course of 
proof of Lemma 4.5 from \cite{IMT-Besov}. We prove the assertion (ii). 
Since $L^2(\Omega) \hookrightarrow \mathcal X' (\Omega)$, the identity \eqref{EQ:p-converge} holds for any $f \in L^p(\Omega)\cap L^2(\Omega)$. 
Then the identity \eqref{EQ:p-converge} holds for any $f \in L^p(\Omega)$ by the density argument, since $1 \le p < \infty$. 
The proof of Lemma \ref{lem:lem-L2} is finished. 
\end{proof}

\subsection{Functional calculus for spectral multipliers}
\label{sec:3.1}
This subsection is devoted to proving 
$L^p$-estimates for the operators $\psi(\mathH)$ and $\phi_j(\sqrt{\mathH})$. \\

We recall the following two results 
from \cite{IMT-bdd}.

\begin{prop}[Theorem 1.1 from \cite{IMT-bdd}]
\label{prop:Lp}
Let $\Omega$ be an open set of $\R^n$. Then
for any 
$\phi \in \mathscr S(\mathbb R)$ and 
$1 \le p\le q\le \infty$, 
there exists a constant $C>0$ such that 
\begin{equation}\label{EQ:Lp}
\|\phi(\theta\mathH)\|_{L^p(\Omega)\to L^q(\Omega)}
\le C\theta^{-\frac{n}{2}\left(\frac{1}{p}-\frac{1}{q}\right)} 
\end{equation}
for any $\theta > 0$. 
\end{prop}

\begin{prop}[Theorem 1.2 from \cite{IMT-bdd}]
\label{prop:grad}
Let $\Omega$ be an open set of $\R^n$. Then 
for any 
$\phi \in \mathscr S(\mathbb R)$ and 
$1 \le p\le 2$,
there exists a constant $C>0$ such that 
\begin{equation*}
\label{EQ:grad}
\|\nabla \phi(\theta \mathH)\|_{L^p(\Omega)\to L^p(\Omega)} 
\le C\theta^{-\frac{1}{2}}
\end{equation*}
for any  $\theta > 0$. 
\end{prop}

As related results of Proposition \ref{prop:grad}, we refer to Coulhon and Duong \cite{CD-1999} and 
Ouhabaz \cite{Ouh_2005}.\\

Based on Proposition \ref{prop:Lp}, we have 
the following.

\begin{lem}\label{lem:Lp}
Let $\Omega$ be an open set of $\R^n$, and 
let $1\le p\le \infty$. Then 
the following assertions hold{\rm :}
\begin{itemize}
\item[(i)] 
For any $m \in \mathbb N \cup \{0\}$ 
there exists a constant $C>0$ such that 
\begin{equation}\label{EQ:Lp1}
\left\| 
\mathH^m \psi(2^{-2j}\mathH)
\right\|_{L^p (\Omega)\to L^p (\Omega)} 
\le 
C 2^{2mj}
\end{equation}
for any $j\in\mathbb Z$.
\item[(ii)]
For any $\alpha \in \mathbb R$ there exists a constant $C>0$ such that 
\begin{equation}\label{EQ:Lp2}
\big\| 
\mathH^{\alpha} \phi_j(\sqrt{\mathH})
\big\|_{L^p (\Omega)\to L^p (\Omega)} 
\le 
C 2^{2\alpha j}
\end{equation}
for any $j\in\mathbb Z$. 
Furthermore, for any $\alpha\ge0$, we have 
\begin{equation}\label{EQ:Lp3}
\Big\| \mathH^\alpha
\sum^j_{k=-\infty}\phi_k(\sqrt{\mathH})
\Big\|_{L^p (\Omega)\to L^p (\Omega)} 
\le C 2^{2\alpha j}
\end{equation}
for any $j\in\mathbb Z$.
\end{itemize}
\end{lem}
\begin{proof}
The estimate \eqref{EQ:Lp1} 
is an immediate consequence of Proposition \ref{prop:Lp}. 
In fact, 
noting that
\[
\lambda^m \psi(\lambda) \in C^\infty_0(\mathbb R),
\]
we conclude from \eqref{EQ:Lp} for $\theta=2^{-2j}$ that
\begin{equation*}
\begin{split}
\left\| 
\mathH^m \psi(2^{-2j}\mathH)
\right\|_{L^p(\Omega)\to L^p(\Omega)}
&=
2^{2mj}
\left\| 
(2^{-2j}\mathH)^m \psi(2^{-2j}\mathH)
\right\|_{L^p(\Omega)\to L^p(\Omega)}\\
&\le
C 2^{2mj}
\end{split}
\end{equation*}
for any $j\in\mathbb Z$.   
In a similar way, we get \eqref{EQ:Lp2}, since 
\[
\lambda^{\alpha} \phi_0(\sqrt{\lambda})
\in C^\infty_0((0,\infty)).
\]
It remains to prove the estimate \eqref{EQ:Lp3}. 
When $\alpha>0$, the estimate \eqref{EQ:Lp3} follows from the estimate \eqref{EQ:Lp2}. 
In fact, we estimate
\begin{equation*}
\begin{split}
\Big\| \mathH^\alpha
\sum^j_{k=-\infty}\phi_k(\sqrt{\mathH})
\Big\|_{L^p (\Omega)\to L^p (\Omega)} 
&\le 
\sum^j_{k=-\infty}
\| \mathH^\alpha
\phi_k(\sqrt{\mathH})
\|_{L^p (\Omega)\to L^p (\Omega)}\\
&\le
C
\sum^j_{k=-\infty}2^{2\alpha k}\\
&\le
C 2^{2\alpha j}.
\end{split}
\end{equation*}
Let us now prove the case when $\alpha=0$. 
It follows from Lemma \ref{lem:lem-L2} 
that  
\begin{equation*}\label{EQ:IMP}
\sum^j_{k=-\infty}\phi_k(\sqrt{\mathH})f
=
\psi(2^{-2j}\mathH)f\quad \text{in }L^2(\Omega)
\end{equation*}
for any $j\in \mathbb Z$ and $f\in L^2(\Omega)$, 
which implies that 
\begin{equation*}
\Big\| 
\sum^j_{k=-\infty}\phi_k(\sqrt{\mathH})g
\Big\|_{L^p(\Omega)}
=
\big\| 
\psi(2^{-2j}\mathH)g
\big\|_{L^p(\Omega)}
\le
C 
\|g\|_{L^p(\Omega)}
\end{equation*}
for any $j\in \mathbb Z$ and $g\in L^p(\Omega)\cap L^2(\Omega)$. 
Thus, when $1 \le p <\infty$, 
the estimate \eqref{EQ:Lp3} for $\alpha=0$ is proved by the density argument, and 
the case $p=\infty$ is obtained from $L^1$-estimate by the duality argument. 
Thus the estimate \eqref{EQ:Lp3} for $\alpha=0$ is proved. 
The proof of Lemma \ref{lem:Lp} is finished.
\end{proof}

Based on the gradient estimate \eqref{EQ:grad-infty} and Proposition \ref{prop:grad}, 
we have the following estimates which play a
crucial role in the proof of Theorem \ref{thm:bilinear}.

\begin{lem}\label{lem:nabla}
Let $1 \le p\le \infty$. Then the following assertions hold{\rm :}
\begin{itemize}
\item[(i)] 
Assume that $\Omega$ is an open set of $\mathbb R^n$ such that  
\eqref{EQ:grad-infty} holds for any $t \in(0,1]$. 
Then for any $m\in\mathbb N\cup\{0\}$ and $\alpha\in\mathbb R$ there exists a constant $C>0$ such that 
\begin{equation}\label{EQ:nabla1}
\|
\nabla \mathH^m \psi(2^{-2j}\mathH)
\|_{L^p(\Omega)\to L^p(\Omega)}
\le C
2^{(2m+1)j},
\end{equation}
\begin{equation}\label{EQ:nabla2}
\|
\nabla \mathH^{\alpha}\phi_j(\sqrt{\mathH})
\|_{L^p(\Omega)\to L^p(\Omega)}
\le C
2^{(2\alpha+1)j}
\end{equation}
for any $j\in \mathbb N$. 
\item[(ii)]
Assume that $\Omega$ is an open set 
of $\mathbb R^n$ such that  
\eqref{EQ:grad-infty} holds for any $t>0$. 
Then 
the estimates \eqref{EQ:nabla1} and \eqref{EQ:nabla2} hold for any $j \in \mathbb Z$. 
Furthermore, for any $\alpha\ge0$ there exists a constant $C>0$ such that
\begin{equation}\label{EQ:nabla3}
\Big\|
\nabla \mathH^\alpha
\sum^j_{k=-\infty}\phi_k(\sqrt{\mathH})
\Big\|_{L^p (\Omega)\to L^p (\Omega)} 
\le C 2^{(2\alpha+1)j}
\end{equation}
for any $j\in\mathbb Z$.
\end{itemize}
\end{lem}
\begin{proof}
We prove the assertion (i). 
The case $p=1$ is an immediate consequence of 
Proposition \ref{prop:grad} for $\theta = 2^{-2j}$, 
since 
\[
\lambda^m\psi(\lambda) \in C^\infty_0(\mathbb R), \quad
\lambda^{\alpha}\phi_0(\sqrt{\lambda}) \in C^\infty_0((0,\infty)).
\]
Hence 
it suffices to show the case $p=\infty$. 
In fact, 
once the case $p=\infty$ is proved, 
the Riesz-Thorin interpolation theorem 
allows us to conclude the estimates \eqref{EQ:nabla1} and \eqref{EQ:nabla2} 
for any $1\le p\le \infty$.

Let $f \in L^\infty(\Omega)$. 
Then it follows from the estimate \eqref{EQ:grad-infty} for $0<t\le 1$ that
\begin{equation}\label{a}
\begin{split}
\left\|
\nabla \mathH^m\psi(2^{-2j}\mathH) f
\right\|_{L^\infty(\Omega)}
&= 
\big\|
\nabla e^{-2^{-2j}\mathH}e^{2^{-2j}\mathH}\mathH^m\psi(2^{-2j}\mathH) f
\big\|_{L^\infty(\Omega)}\\
&\le 
C 2^j
\big\|
e^{2^{-2j}\mathH}\mathH^m\psi(2^{-2j}\mathH) f
\big\|_{L^\infty(\Omega)}\\
&=
C 2^{(2m+1)j}
\big\|
e^{2^{-2j}\mathH}(2^{-2j}\mathH)^m\psi(2^{-2j}\mathH) f
\big\|_{L^\infty(\Omega)}
\end{split}
\end{equation}
for any $j\in\mathbb N$. 
Since
\[
e^{\lambda}\lambda^m\psi(\lambda) \in C^\infty_0(\mathbb R),
\]
it follows from the estimate \eqref{EQ:Lp} for $p=\infty$ in Proposition \ref{prop:Lp} that 
\begin{equation}\label{b}
\big\|
e^{2^{-2j}\mathH}(2^{-2j}\mathH)^m\psi(2^{-2j}\mathH) f
\big\|_{L^\infty(\Omega)}
\le C
\|f\|_{L^\infty(\Omega)}.
\end{equation}
Thus the required estimate \eqref{EQ:nabla1} for $p=\infty$ is an immediate consequence of \eqref{a} and \eqref{b}. 
In a similar way, we get \eqref{EQ:nabla2}. 
Thus the assertion (i) is proved. 

Next we prove the assertion (ii). 
We can prove the estimates \eqref{EQ:nabla1} and \eqref{EQ:nabla2} for any $j \in \mathbb Z$ in the same way as (i).
Furthermore, 
the estimate \eqref{EQ:nabla3} is proved by using \eqref{EQ:nabla2} 
in the same way as the proof of 
\eqref{EQ:Lp2} for $\alpha>0$. 
Hence we may omit the details. 
The proof of Lemma \ref{lem:nabla} is finished.
\end{proof}

\subsection{The Leibniz rule for the Dirichlet Laplacian}
\label{sec:3.2}
In this subsection 
we prove the following lemma.  

\begin{lem}
\label{lem:Leib}
 Assume that $\Omega$ is an open set of 
$\R^n$ such that \eqref{EQ:grad-infty} holds for any $t\in(0,1]$. 
Let 
$\Phi, \Psi \in C^\infty_0(\mathbb R)$.  
Then for any $f,g\in \mathcal X'(\Omega)$,
we have 
\begin{equation}
\label{EQ:Leib}
\begin{split}
&\mathH  \big( \Phi(\mathH)f \cdot \Psi(\mathH)g \big)\\
= &\, 
\mathH  \Phi (\mathH)f\cdot
\Psi(\mathH)g
-
2\nabla \Phi(\mathH)f \cdot \nabla \Psi(\mathH)g
+
\Phi(\mathH)f \cdot\mathH \Psi(\mathH)g
\quad \text{in }\mathcal X'(\Omega).
\end{split}
\end{equation}
\end{lem}
\begin{proof}
To begin with, 
we note from Lemma \ref{lem:decomposition1} that 
$\Phi(\mathH)f$ and $\Psi(\mathH)g$ are regarded as elements in $L^\infty(\Omega)$:
\begin{equation}
\label{EQ:Linfty}
\Phi(\mathH)f, \, \Psi(\mathH)g\in L^\infty(\Omega).
\end{equation}
Noting that the assumption \eqref{EQ:grad-infty} is necessary for 
Lemma \ref{lem:nabla}, we  apply Lemmas \ref{lem:Lp} and \ref{lem:nabla} for $p=\infty$.
Then we see that
\begin{equation}
\label{EQ:g-Linfty}
\mathH\Phi(\mathH)f, \, \mathH\Psi(\mathH)g, \, \nabla \Phi(\mathH)f,  \, \nabla \Psi(\mathH)g
\in L^\infty(\Omega).
\end{equation}
Hence, all terms on the right hand side of \eqref{EQ:Leib}
belong to $L^\infty(\Omega)$.
Therefore, it suffices to show that 
\eqref{EQ:Leib} holds in $\mathscr{D}^\prime(\Omega)$, where 
$\mathscr{D}^\prime(\Omega)$ is the space consisting of distributions on $\Omega$, 
i.e., the dual space of $\mathscr{D}(\Omega)$.
In fact, if \eqref{EQ:Leib} holds in $\mathscr D'(\Omega)$, then 
\eqref{EQ:Leib} holds 
almost everywhere on 
$\Omega$.  
Thus we conclude that \eqref{EQ:Leib} holds 
in $\mathcal{X}^\prime(\Omega)$. 

Since 
\[
\mathcal{H}h=-\Delta h 
\quad \text{for $h\in \mathscr D(\Omega)$,}
\]
we write, by using \eqref{EQ:Linfty},
\begin{equation}
\label{EQ:a}
\begin{split}
{}_{\mathscr D'(\Omega)}\langle 
\mathH \left(\Phi(\mathH)f \cdot 
\Psi(\mathH)g\right),h
\rangle_{\mathscr D(\Omega)}
= &
{}_{L^\infty(\Omega)}\langle 
\Psi(\mathH)g,
\overline{\Phi(\mathH)f} (-\Delta h)
\rangle_{L^1(\Omega)}
\end{split}
\end{equation}
for any $h\in \mathscr{D}(\Omega)$. 
Here, noting from the definition \eqref{EQ:X'} of 
$\mathH$ that
\[
-\Delta \Phi(\mathcal{H})f=\mathcal{H}\Phi(\mathcal{H})f\quad \text{in }\mathscr D'(\Omega),
\]
we observe from 
the Leibniz rule that
\begin{equation}
\label{EQ:Leib3}
\overline{\Phi(\mathH)f}(-\Delta h)
=-\Delta(\overline{\Phi(\mathH)f}\cdot h)-(\overline{\mathH  \Phi(\mathH)f})h
+2\overline{\nabla \Phi(\mathH)f}\cdot\nabla h
\quad \text{in $\mathscr D'(\Omega)$.}
\end{equation}
Since all the terms in \eqref{EQ:Leib3} belong to $L^1(\Omega)$ by \eqref{EQ:Linfty} 
and \eqref{EQ:g-Linfty}, 
multiplying \eqref{EQ:Leib3} by 
$\Psi(\mathH)g$, and using \eqref{EQ:a},
we write 
\begin{equation}
\label{EQ:b}
\begin{split}
&
{}_{\mathscr D'(\Omega)}\langle 
\mathH \left(\Phi(\mathH)f \cdot\Psi(\mathH)g
\right),
h
\rangle_{\mathscr D(\Omega)}\\
= &\,
{}_{L^\infty(\Omega)}\langle 
\Psi(\mathH)g,
-\Delta(\overline{\Phi(\mathH)f}\cdot h)
\rangle_{L^1(\Omega)}\\
& \quad - 
{}_{L^\infty(\Omega)}\langle
(\mathH \Phi(\mathH)f)\Psi(\mathH)g, h
\rangle_{L^1(\Omega)}
+2
{}_{L^\infty(\Omega)}\langle
\Psi(\mathH)g,
\overline{\nabla \Phi(\mathH)f}\cdot\nabla h
\rangle_{L^1(\Omega)}.
\end{split}
\end{equation}
As to the first term in the right member of \eqref{EQ:b}, integrating by parts, we get
\begin{equation*}
\begin{split}
{}_{L^\infty(\Omega)}\langle 
\Psi(\mathH)g,
-\Delta(\overline{\Phi(\mathH)f}\cdot h)
\rangle_{L^1(\Omega)}
= 
{}_{L^\infty(\Omega)}\langle 
-\Delta \Psi(\mathH)g,
\overline{\Phi(\mathH)f}\cdot h
\rangle_{L^1(\Omega)}. 
\end{split}
\end{equation*}
Here, we note that 
\begin{equation}
\label{EQ:c}
-\Delta \Psi(\mathH)g 
= \mathH \Psi(\mathH)g
\quad \text{in }\mathscr D'(\Omega).
\end{equation}
Since $\mathH \Psi(\mathH)g$ belongs to $L^\infty(\Omega)$ 
by \eqref{EQ:Linfty} and 
Lemma \ref{lem:Lp} for $p=\infty$, 
the identity \eqref{EQ:c} holds  
almost everywhere on $\Omega$. 
Hence we have
\begin{equation*}
{}_{L^\infty(\Omega)}\langle 
-\Delta \Psi(\mathH)g,
\overline{\Phi(\mathH)f}\cdot h
\rangle_{L^1(\Omega)}
=
{}_{L^\infty(\Omega)}\langle 
\Phi(\mathH)f \cdot \mathH \Psi(\mathH)g,
h
\rangle_{L^1(\Omega)},
\end{equation*}
since $\overline{\Phi(\mathH)f}\cdot 
h \in L^1(\Omega)$. 
Therefore, the first term is written as 
\begin{equation*}
\label{EQ:d}
{}_{L^\infty(\Omega)}\langle 
\Psi(\mathH)g,
-\Delta (\overline{\Phi(\mathH)f}\cdot 
h)
\rangle_{L^1(\Omega)}
=
{}_{L^\infty(\Omega)}\langle 
\Phi(\mathH)f\cdot \mathH \Psi(\mathH)g,
h
\rangle_{L^1(\Omega)}.
\end{equation*}
In a similar way, the third term 
in the right member of \eqref{EQ:b}
is written as 
\begin{equation}\label{EQ:Last}
\begin{split}
&{}_{L^\infty(\Omega)}\langle 
\Psi(\mathH)g,
\overline{\nabla \Phi(\mathH)f}\cdot\nabla h
\rangle_{L^1(\Omega)}\\
= &
-{}_{\mathscr D'(\Omega)}\langle 
\Delta \Phi(\mathH)f \cdot 
\Psi(\mathH)g,h
\rangle_{\mathscr D(\Omega)}
-{}_{\mathscr D'(\Omega)}\langle 
\nabla\Phi(\mathH)f \cdot \nabla\Psi(\mathH)g,
h
\rangle_{\mathscr D(\Omega)}\\
= &
{}_{\mathscr D'(\Omega)}\langle 
\mathH \phi(\mathH)f \cdot 
\Psi(\mathH)g,
h
\rangle_{\mathscr D(\Omega)}
-
{}_{\mathscr D'(\Omega)}\langle 
\nabla \Phi(\mathH)f\cdot\nabla \Psi(\mathH)g,
h
\rangle_{\mathscr D(\Omega)}.
\end{split}
\end{equation}
Therefore, summarizing \eqref{EQ:b} and 
\eqref{EQ:Last}, 
we conclude that \eqref{EQ:Leib} holds 
in $\mathscr{D}^\prime(\Omega)$. 
The proof of Lemma \ref{lem:Leib} is finished.
\end{proof}

\subsection{Properties of the space ${\mathcal P}(\Omega)$}
\label{sec:3.3}

In this subsection we shall study several 
properties of a space ${\mathcal P}(\Omega)$, which is defined by
\begin{equation}
\label{EQ:P}
{\mathcal P}(\Omega)
:=
\left\{f\in{\mathcal X}'(\Omega)\, \Big|\, 
{}_{{\mathcal Z}'(\Omega)}\langle 
f, g
\rangle_{{\mathcal Z}(\Omega)}
=
0 
\text{ for any $g\in{\mathcal Z}(\Omega)$}
\right\}.
\end{equation}

\vspace{5mm}

We recall that ${\mathcal X}'(\Omega)$ and $\mathcal Z'(\Omega)$ 
correspond to  
$\mathscr S'(\mathbb R^n)$ and $\mathscr S'_0(\mathbb R^n)$, respectively. 
Here $\mathscr S'_0(\mathbb R^n)$ is the dual space of $\mathscr S_0(\mathbb R^n)$ defined by 
\[
\mathscr S_0(\mathbb R^n)
:=
\left\{ f\in \mathscr S(\mathbb R^n)\,\Big|\,
\int_{\mathbb R^n} x^\alpha f(x)\,dx = 0
\quad \text{for any $\alpha \in (\mathbb N \cup \{0\})^n$}\right\}
\]
endowed with the induced topology of $\mathscr S(\mathbb R^n)$. 
It is known that $\mathscr S'_0(\mathbb R^n)$ is characterized by the quotient space of $\mathscr S'(\mathbb R^n)$ modulo polynomials, i.e., 
\begin{equation*} 
\mathscr S'_0(\mathbb R^n) \cong 
\mathscr S'(\mathbb R^n) / \mathcal P,
\end{equation*}
where $\mathcal P$ is the set of all polynomials of $n$ real variables (see, e.g.,  Proposition 1.1.3 from Grafakos \cite{Grafakos_2014}). 
As to the space $\mathcal P (\Omega)$, 
it is readily checked that $\mathcal P (\Omega)$ is a closed subspace of $\mathcal X'(\Omega)$, and hence, we can apply Theorem in p.126 from Schaefer \cite{Sch_1971} and Propositions 35.5 and 35.6 from Tr\`eves \cite{Tre_1967} to obtain 
the isomorphism:
\[
\mathcal Z'(\Omega) \cong \mathcal X'(\Omega)/\mathcal P(\Omega)
\]
(cf. Theorem 1.1 from Sawano \cite{Saw-2016}).\\

We shall prove the following:
\begin{lem}
\label{lem:P}
Let $\Omega$ be an open set of $\R^n$.  
Then the space $\mathcal P(\Omega)$ enjoys the following{\rm :}
\begin{itemize}
\item[(i)] Let $f \in \mathcal X'(\Omega)$. 
Then the following assertions are equivalent{\rm :}
\begin{itemize}
\item[(a)] $f\in \mathcal P(\Omega)${\rm ;}
\item[(b)] 
$\phi_j(\sqrt{\mathH})f = 0$ in $\mathcal X'(\Omega)$ for any $j\in \mathbb Z${\rm ;}
\item[(c)] 
$\|f\|_{\dot{B}^s_{p,q}(\mathH)} = 0$ for 
any $s \in \mathbb R$ and $1 \le p,q \le \infty$.
\end{itemize}
\item[(ii)] 
$\mathcal P(\Omega)$ is a subspace of $L^\infty(\Omega)$. 
\item[(iii)] In particular, if $\Omega$ is 
a domain such that \eqref{EQ:grad-infty} holds for any $t>0$, then
\[
\mathcal P(\Omega) = 
\text{either} \quad 
\{0\} \quad \text{or}\quad \{f = c \text{ on }\Omega \,|\, c \in \mathbb C \}.
\]
\end{itemize}
\end{lem}
\begin{proof}
We prove the assertion (i). It is readily 
seen from the definition of 
$\dot{B}^s_{p,q}(\mathH)$
that (c) implies (b), since 
\[
\text{$\phi_j(\sqrt{\mathH})f=0$ \quad in 
$L^p(\Omega)$}
\]
for any $j\in \mathbb{Z}$, and since 
$L^p(\Omega)\hookrightarrow \mathcal X^\prime(\Omega)$.
Conversely, we suppose that (b) holds. 
Since 
$f\in \mathcal{X}^\prime(\Omega)$, 
it follows from part (i) of Lemma \ref{lem:decomposition1}
that 
\[
\phi_j(\sqrt{\mathH})f\in L^\infty(\Omega)
\]
for any $j\in \mathbb{Z}$.
Hence, thanks to fundamental lemma of the 
calculus of variations, we deduce that
\[
\phi_j(\sqrt{\mathH})f(x)=0 \quad 
\text{a.e. $x\in \Omega$}
\]
for any $j\in \mathbb Z$.
which implies that (c) holds true. 

We have to prove that
(a) and (b) are equivalent. Suppose that 
(a) holds, i.e., $f \in \mathcal P(\Omega)$. 
We note that if $g\in\mathcal X(\Omega)$,
then 
\begin{equation}\label{EQ:11-KK}
\text{$\phi_j(\sqrt{\mathH})g\in \mathcal{Z}(\Omega)$ \quad 
for any $j\in \mathbb Z$. }
\end{equation}
In fact, fixing $j\in \mathbb{Z}$, we have 
\[
\phi_k(\sqrt{\mathcal{H}}) \phi_j(\sqrt{\mathcal{H}})g\ne 0 
\]
only if $k=j-1,j,j+1$. Then, by using 
Proposition \ref{prop:Lp},
we deduce that for any 
$M\in \mathbb{N}$,
\begin{equation*}
\begin{split}
\sup_{k \le 0}2^{-Mk}\|\phi_k(\sqrt{\mathH})\phi_j(\sqrt{\mathH})g\|_{L^1(\Omega)}
\le &\,
C\max_{k=j-1,j,j+1}2^{-Mk}\|\phi_j(\sqrt{\mathH})g\|_{L^1(\Omega)}\\
\le &\,
C2^{-Mj}\|\phi_j(\sqrt{\mathH})g\|_{L^1(\Omega)}\\
\le & \,
C2^{-Mj}\|g\|_{L^1(\Omega)}\\
<&\infty,
\end{split}
\end{equation*}
which implies \eqref{EQ:11-KK}.
Since $f\in \mathcal{P}(\Omega)$,
thanks to \eqref{EQ:11-KK}, 
it follows that
\[
{}_{\mathcal X' (\Omega)}\langle 
\phi_j(\sqrt{\mathH})f, g
\rangle_{\mathcal X(\Omega)}
=
{}_{\mathcal Z'(\Omega)}\langle 
f, \phi_j(\sqrt{\mathH})g
\rangle_{\mathcal Z(\Omega)}
=
0
\]
for any $j\in \mathbb{Z}$ and $g\in\mathcal X(\Omega)$, 
which implies (b). 
Conversely, let us suppose that 
(b) holds. Since $\mathcal{Z}(\Omega)\subset 
\mathcal{X}(\Omega)$, it follows that 
\begin{equation}\label{EQ:z}
{}_{\mathcal Z'(\Omega)}\langle \phi_j(\sqrt{\mathH})f,g\rangle_{\mathcal Z(\Omega)}=
{}_{\mathcal X'(\Omega)}\langle \phi_j(\sqrt{\mathH})f,g\rangle_{\mathcal X(\Omega)}= 0
\end{equation}
for any $j\in \mathbb{Z}$ and 
$g\in \mathcal{Z}(\Omega)$.
Here, we recall part (ii) of Lemma 
\ref{lem:decomposition1} that 
\[
f=\sum_{j=-\infty}^\infty \phi_j(\sqrt{\mathcal{H}})f
\quad \text{in $\mathcal{Z}'(\Omega)$}.
\]
Then, by using this identity and \eqref{EQ:z}, we have 
\[
{}_{\mathcal Z'(\Omega)}\langle 
f, g
\rangle_{\mathcal Z(\Omega)}
=
\sum_{j=-\infty}^\infty {}_{\mathcal Z'(\Omega)}\langle 
\phi_j(\sqrt{\mathcal{H}})f, g
\rangle_{\mathcal Z(\Omega)}
=
0
\]
for any $g\in\mathcal Z(\Omega)$, 
which implies that $f\in \mathcal{P}(\Omega)$.
Hence (a) holds true. 
Thus we conclude the assertion (i).

Next we prove the assertion (ii). 
Let $f \in \mathcal P(\Omega)$. 
It follows from \eqref{907-1} in Lemma 
\ref{lem:decomposition1} that  
\[
f = \psi(\mathH)f +\sum_{j=1}^\infty
\phi_j(\sqrt{\mathcal{H}})f 
\quad \text{in $\mathcal X'(\Omega)$}. 
\]
Applying (b) in the assertion (i) to the 
second term in the right member, we get 
\begin{equation}\label{EQ:FF}
f =\psi(\mathH)f 
\quad \text{in $\mathcal X'(\Omega)$}. 
\end{equation}
Since $\psi(\mathH)f \in L^\infty(\Omega)$ 
by the assertion
(i) in Lemma \ref{lem:decomposition1}, 
we conclude that $f \in L^\infty(\Omega)$. 
Therefore, the assertion (ii) is proved. 

Finally we show the assertion (iii). 
Let $f\in\mathcal P(\Omega)$. Then, 
again by using the argument in \eqref{EQ:FF}, 
we see that
\begin{equation}\label{EQ:GG}
f = \psi(2^{-2k}\mathH)f +\sum_{j=k}^\infty
\phi_j(\sqrt{\mathcal{H}})f= 
\psi(2^{-2k}\mathH)f
\quad \text{in $\mathcal X'(\Omega)$} 
\end{equation}
for any $k\in\mathbb Z$. 
Since the gradient estimate \eqref{EQ:grad-infty} holds for $t=2^{-2k}$, applying 
\eqref{EQ:nabla1} from Lemma \ref{lem:nabla}
to the last member in \eqref{EQ:GG}, we get 
\begin{equation*}
\begin{split}
\|\nabla f\|_{L^\infty(\Omega)}
= &
\|\nabla \psi(2^{-2k}\mathH)f\|_{L^\infty(\Omega)}\\
\le & C 
2^k\|f\|_{L^\infty(\Omega)}
\end{split}
\end{equation*}
for any $k\in\mathbb Z$, which 
implies that $\nabla f = 0$ in $\Omega$. 
Since $\Omega$ is connected, 
$f$ is a constant in $\Omega$. 
Summarizing the above argument, we deduce 
that 
\[
\{0\}\subset \mathcal P(\Omega)
\subset \{f = c \text{ on }\Omega \,|\, c \in \mathbb C \}.
\]
Since $\mathcal P(\Omega)$ is a linear space, we conclude that if $\mathcal P(\Omega)\ne\{0\}$, then $\mathcal P(\Omega)$ is the space
of all constant functions on $\Omega$. 
This proves (iii). 
The proof of 
Lemma \ref{lem:P} is finished. 
\end{proof}

\subsection{A lemma on convergence in Besov spaces}
In this subsection we shall prove 
a lemma in Besov spaces, which is useful in 
the proof of the theorem.  
\begin{lem}\label{lem:Fatou}
Let $\Omega$ be an open set of $\R^n$,
and let $s \in \mathbb R$ and $1 \le p,q \le \infty$. 
Assume that $\{f_N\}_{N\in\mathbb N}$ is a 
bounded sequence in $\dot{B}^s_{p,q}(\mathH)$, and that 
there exists an $f \in \mathcal X'(\Omega)$ such that 
\begin{equation}\label{EQ:convergence}
f_N \to f \quad \text{in $\mathcal X'(\Omega)$ as $N\to\infty$}.
\end{equation}
Then $f \in \dot{B}^s_{p,q}(\mathH)$ and 
\begin{equation}\label{EQ:Fatou}
\|f\|_{\dot{B}^s_{p,q}(\mathH)}
\le \liminf_{N\to \infty}\|f_N\|_{\dot{B}^s_{p,q}(\mathH)}.
\end{equation}
\end{lem}

Before going to the proof, 
let us give a remark on the idea of 
proof of the lemma. 
When $1 < p,q < \infty$, 
$\dot{B}^s_{p,q}(\mathH)$ are reflexive for any $s\in \R$. 
This fact and the limiting properties of the weak convergence 
imply the inequality \eqref{EQ:Fatou}. 
Otherwise, we need the pointwise convergence 
of $\phi_j (\sqrt{\mathH})f_N$, 
which is obtained directly with a property of the kernel 
$\va(\mathcal{H})(x,y)$ of the operator 
$\va(\mathcal{H})$. 
Let us investigate the property of the kernel. 

\begin{lem}\label{lem:kernel}
Let $\Omega$ be an open set of $\R^n$. 
Then for any $\phi\in\mathscr S(\mathbb R)$, 
we have 
\begin{equation}\label{EQ:kernel}
\phi(\mathH)(x,\cdot) \in \mathcal X(\Omega)
\quad \text{for each $x \in\Omega$}. 
\end{equation}
\end{lem}
\begin{proof}
Since 
\[
\|\phi(\mathH)\|_{L^p(\Omega)\to L^\infty(\Omega)} < \infty
\]
for any $1 \le p \le \infty$ by  
Proposition \ref{prop:Lp},
it follows from Lemma \ref{lem:norm} 
in appendix \ref{App:AppendixB} that 
\[
\sup_{x\in\Omega}
\|\phi(\mathH)(x,\cdot)\|_{L^{p'}(\Omega)}
=
\|\phi(\mathH)\|_{L^p(\Omega)\to L^\infty(\Omega)} 
\]
for any $1 \le p \le \infty$, where $p'$ is 
the conjugate exponent of $p$.
Hence we have 
\begin{equation}\label{EQ:kernel2}
\phi(\mathH)(x,\cdot) \in L^{p^\prime}(\Omega)
\quad \text{for each $x \in\Omega$}.
\end{equation}
In particular, we have 
$$\mathH^M(\phi(\mathH)(x,\cdot))
\in \mathcal X'(\Omega)
$$
for any $M \in \mathbb N$, since 
$L^{p'}(\Omega)\hookrightarrow \mathcal X^\prime(\Omega)$, and since $\mathH^M$ maps 
$\mathcal X^\prime(\Omega)$ to itself.
We denote by $K_{\mathH^M\phi(\mathH)}(x,y)$ the 
kernel of $\mathH^M\phi(\mathH)$. 
Then, for any $f \in \mathcal X(\Omega)$, we have
\begin{equation*}
\begin{split}
{}_{\mathcal X'(\Omega)}\langle \mathH^M(\phi(\mathH)(x,\cdot)), f \rangle_{\mathcal X(\Omega)}
& = {}_{\mathcal X'(\Omega)}\langle \phi(\mathH)(x,\cdot),\mathH^M f \rangle_{\mathcal X(\Omega)}\\
& =\phi(\mathH)\mathH^M \overline{f}(x)\\
& = \mathH^M \phi(\mathH)\overline{f}(x)\\
& = 
{}_{\mathcal X'(\Omega)}\langle K_{\mathH^M \phi(\mathH)}(x,\cdot), f \rangle_{\mathcal X(\Omega)}
\end{split}
\end{equation*}
for any $x\in\Omega$, 
which implies that
\[
\mathH^M(\phi(\mathH)(x,\cdot))(y)
=
K_{\mathH^M\phi(\mathH)}(x,y)
\quad \text{a.e.}\,y\in\Omega
\]
for any $x\in\Omega$.
Since 
\[
\lambda^M\phi(\lambda) \in \mathscr S(\mathbb R)
\]
for any $M \in \mathbb N$, it follows from \eqref{EQ:kernel2} 
for $p'=1$ and $p'=2$
that 
\[
K_{\mathH^M\phi(\mathH)}(x,\cdot) \in L^1(\Omega)\cap L^2(\Omega)
\]
for any $M\in\mathbb N$ and $x\in\Omega$. 
Hence we obtain
\[
\mathH^M(\phi(\mathH)(x,\cdot)) \in L^1(\Omega)\cap L^2(\Omega)
\]
for any $M\in\mathbb N$ and $x\in\Omega$. Thus we conclude \eqref{EQ:kernel}. The proof of Lemma \ref{lem:kernel} is finished. 
\end{proof}

We are in a position to prove Lemma \ref{lem:Fatou}.
\begin{proof}[Proof of Lemma \ref{lem:Fatou}]
First, we show that 
\begin{equation}\label{EQ:p-con}
\phi_j(\sqrt{\mathH})f_N (x)\to \phi_j(\sqrt{\mathH})f (x)
\quad \text{a.e.}\,x\in\Omega\text{ as $N\to\infty$}
\end{equation}
for each $j\in\mathbb Z$.
Put $$\Phi_j = \phi_{j-1} +\phi_j +\phi_{j+1}$$ 
for $j \in \mathbb Z$. 
Then, noting from the assertion (i) in Lemma \ref{lem:decomposition1} that
\[
\Phi_j (\sqrt{\mathH}) f_N
\in L^\infty(\Omega),
\]
and from Lemma \ref{lem:kernel} that
\[
\phi_j(\sqrt{\mathH})(x,\cdot) \in \mathcal X(\Omega)
\quad \text{for each $x \in\Omega$}, 
\]
we write
\begin{equation}\label{EQ:3.5-1}
\begin{split}
\phi_j(\sqrt{\mathH})f_N (x) & =\phi_j(\sqrt{\mathH})\Phi_j 
(\sqrt{\mathH})f_N (x)\\
& =
{}_{\mathcal X'(\Omega)}\langle \Phi_j (\sqrt{\mathH}) f_N, \phi_j(\sqrt{\mathH})(x,\cdot)\rangle_{\mathcal X(\Omega)}
\end{split}
\end{equation}
for each $j\in\mathbb Z$ and $x \in \Omega$.
In a similar way, we have 
\begin{equation}\label{EQ:3.5-2}
\phi_j(\sqrt{\mathH})f (x) = 
{}_{\mathcal X'(\Omega)}\langle \Phi_j (\sqrt{\mathH}) f, \phi_j(\sqrt{\mathH})(x,\cdot)\rangle_{\mathcal X(\Omega)}
\end{equation}
for each $j\in\mathbb Z$ and $x \in \Omega$. 
Since 
\[
\Phi_j (\sqrt{\mathH})f_N \to \Phi_j (\sqrt{\mathH})f \quad \text{in $\mathcal X'(\Omega)$ as $N\to\infty$}
\]
for each $j\in\mathbb Z$ by assumption  \eqref{EQ:convergence} and the continuity of 
$\Phi_j (\sqrt{\mathH})$ from $\mathcal{X}'(\Omega)$ into 
itself, 
we deduce that 
\begin{equation}\label{EQ:3.5-3}
{}_{\mathcal X'(\Omega)}\langle \Phi_j (\sqrt{\mathH}) f_N, \phi_j(\sqrt{\mathH})(x,\cdot)\rangle_{\mathcal X(\Omega)}
\to 
{}_{\mathcal X'(\Omega)}\langle \Phi_j (\sqrt{\mathH}) f, \phi_j(\sqrt{\mathH})(x,\cdot)\rangle_{\mathcal X(\Omega)}
\end{equation}
for each $j\in\mathbb Z$ and $x \in \Omega$ as $N \to \infty$. 
Hence, combining \eqref{EQ:3.5-1}, \eqref{EQ:3.5-2} and \eqref{EQ:3.5-3}, we get 
the pointwise convergence \eqref{EQ:p-con}.

Let us turn to the proof of the inequality \eqref{EQ:Fatou}. To begin with, given 
$1\le p\le \infty$, we claim that 
\begin{equation}\label{EQ:Fatou-Lp}
\|\phi_j(\sqrt{\mathH})f\|_{L^p(\Omega)}
\le
\liminf_{N\to \infty}\|\phi_j(\sqrt{\mathH})f_N\|_{L^p(\Omega)}
\end{equation}
for each $j\in \mathbb{Z}$.
When $1\le p<\infty$, the inequality 
\eqref{EQ:Fatou-Lp} is a consequence of 
\eqref{EQ:p-con} and Fatou's lemma.
We have to prove the case when $p=\infty$. 
In this case, thanks to \eqref{EQ:p-con}, 
the inequality \eqref{EQ:Fatou-Lp} is true for $p=\infty$, 
since $\{\phi_j(\sqrt{\mathH})f_N\}_{N\in\mathbb{N}}$
is a bounded sequence in $L^\infty(\Omega)$. 
Finally, multiplying by $2^{sj}$ to the both sides 
of \eqref{EQ:Fatou-Lp}, we conclude the 
required inequality \eqref{EQ:Fatou}. 
The proof of Lemma \ref{lem:Fatou} is finished.
\end{proof}

\section{Proof of Theorem \ref{thm:bilinear}} \label{sec:sec4}
In this section we prove Theorem \ref{thm:bilinear}. In the inhomogeneous case
the approximation of the identity is written as  
\[
f=
\psi(\mathH)f +
\sum_{k=1}^\infty \phi_k(\sqrt{\mathH})f
\quad \text{in $\mathcal{X}^\prime(\Omega)$,}
\]
and in the homogeneous
case we write 
\[
f=\sum_{k=-\infty}^\infty \phi_k(\sqrt{\mathH})f
\quad \text{in $\mathcal{Z}^\prime(\Omega)$.}
\]
Hence it is sufficient to 
prove the homogeneous case (ii), since one can reduce 
the argument of the proof of (i) to that of (ii). Therefore, we shall concentrate on proving 
the case (ii). \\

Hereafter, for the sake of simplicity, 
we use the following notations:  
\[
f_j:=\phi_j(\sqrt{\mathH})f,\quad
S_j(f)=S_j(\sqrt{\mathH})(f)
:=\sum_{k=-\infty}^j \phi_k(\sqrt{\mathH})f, 
\quad j\in \mathbb{Z}.
\]
We have to divide the proof 
into two cases: 
\[
\text{``$1\le p_2,p_3<\infty$" and ``$p_2=\infty$ or $p_3=\infty$",}
\] 
since the approximation by the Littlewood-Paley partition of 
unity is available only for $p_2, p_3 < \infty$ 
(see \eqref{907-2}) and a constant function in $\mathcal P (\Omega)$ 
defined by \eqref{EQ:P} appears in the case when 
$p_2 = \infty$ or $p_3 = \infty$. 

\vspace{5mm}

{\bf The case: $1\le p_2,p_3<\infty$.}
Let $f\in \dot{B}^s_{p_1, q}(\mathH)\cap 
L^{p_3}(\Omega)$ and $g\in \dot{B}^s_{p_4, q}(\mathH)
\cap L^{p_2}(\Omega)$. 
Referring to the Bony paraproduct formula (see \cite{Bo-1981}), 
we write 
\begin{equation*}
fg = \sum_{k \in \mathbb{Z}} f_kS_{k-3}(g) + \sum_{k \in \mathbb{Z}} S_{k-3}(f)g_k + \sum_{k\in\mathbb Z}\sum_{|k - l| \le 2} f_kg_l
\quad \text{in }\mathcal X'(\Omega),
\end{equation*}
which is assured by the assertion (ii) in Lemma \ref{lem:lem-L2},  since $p_2,p_3<\infty$.  
By using Minkowski's inequality, we write
\begin{equation*}
\label{EQ:2}
\|fg\|_{\dot{B}^s_{p,q}(\mathH)}
\le
I+II+III+IV+V+VI,
\end{equation*}
where we put
\[
I:= 
\Biggl\{
\sum_{j \in \mathbb{Z}}
\bigg(2^{sj}\sum_{|k-j| \le 2}
\Big\| \phi_{j}(\sqrt{\mathH})\Big( f_kS_{k-3}(g)\Big) 
\Big\|_{L^p(\Omega)}\bigg)^q 
\Biggr\}^{\frac{1}{q}},
\]
\[
II:= 
\Biggl\{
\sum_{j \in \mathbb{Z}}
\bigg(2^{sj}\sum_{|k-j| > 2}
\Big\| \phi_{j}(\sqrt{\mathH})\Big( f_kS_{k-3}(g)\Big) 
\Big\|_{L^p(\Omega)}\bigg)^q 
\Biggr\}^{\frac{1}{q}},
\]
\[
III:= 
\Biggl\{
\sum_{j \in \mathbb{Z}}
\bigg(2^{sj}\sum_{|k-j| \le 2}
\Big\| \phi_{j}(\sqrt{\mathH})\Big( S_{k-3}(f)g_k\Big) 
\Big\|_{L^p(\Omega)}\bigg)^q 
\Biggr\}^{\frac{1}{q}},
\]
\[
IV:= 
\Biggl\{
\sum_{j \in \mathbb{Z}}
\bigg(2^{sj}\sum_{|k-j| > 2}
\Big\| \phi_{j}(\sqrt{\mathH})\Big( S_{k-3}(f)g_k\Big) 
\Big\|_{L^p(\Omega)}\bigg)^q 
\Biggr\}^{\frac{1}{q}},
\]
\[
V:= 
\Biggl\{
\sum_{j \in \mathbb{Z}}
\bigg(
2^{sj}
\sum_{ \substack{ j-2 \le k \\ {\rm or} \\ j-2 \le l } }
\Big\| 
\phi_{j}(\sqrt{\mathH})
\Big(
\sum_{|k-l| \le 2} f_k g_l
\Big) 
\Big\|_{L^p(\Omega)} 
\bigg)^q
\Biggr\}^{\frac{1}{q}},
\]
\[
VI:= 
\Biggl\{
\sum_{j \in \mathbb{Z}}
\bigg(
2^{sj}
\sum_{ \substack{ j-2 > k \\ {\rm and} \\ j-2 > l } }
\Big\| 
\phi_{j}(\sqrt{\mathH})
\Big(
\sum_{|k-l| \le 2} f_k g_l
\Big) 
\Big\|_{L^p(\Omega)} 
\bigg)^q
\Biggr\}^{\frac{1}{q}}.
\]
We note that when $\Omega=\R^n$, 
the terms $II$, $IV$ and $VI$ vanish. 
Indeed, the integrand of 
term $II$ is written as 
the inverse Fourier transform of 
the product of $\phi_j(|\xi|)$ with 
the 
convolution product 
of $\phi_k(|\xi|)\mathscr F f$ and 
$S_{k-3}(|\xi|)\mathscr F g$, where 
$\mathscr{F}f$ denotes the Fourier transform of $f$. Since the supports of 
$\phi_j(|\xi|)$ do not intersect with those 
of the previous convolution product 
for $|j-k|>2$, we deduce that $II$ vanishes. 
In a similar way we find that 
$IV$ and $VI$ also vanish. 
However, when $\Omega\ne\R^n$, 
the integrands do not vanish in general. 
In fact, if $II$, $IV$ and $VI$ vanish, 
the bilinear estimates hold for all positive regularity 
$s$ by the argument of Case A below. 
However it contradicts the 
counter-example constructed in appendix~A. 
It should be noted that the assumption 
\eqref{EQ:grad-infty} on the gradient estimate
plays an 
essential role in the estimation of these terms 
$II$, $IV$ and $VI$.\\

Thus we estimate separately as follows: 
\[
\text{Case A: Estimates for $I, III$ and $V$ \quad \text{and} \quad 
Case B: Estimates for 
$II, IV$ and $VI$. }
\]

\vspace{5mm}

{\bf Case A: Estimates for 
$I, III$ and $V$.}
These terms can be estimated in the same way as in the case when 
$\Omega = \mathbb R^n$. Since similar 
arguments also appear for $II, IV$ 
and $VI$, we give the proof 
in a self-contained way.
First we estimate the term $I$. 
Noting from the assertion (ii) in Lemma \ref{lem:Lp} that
$f_k\in L^{p_1}(\Omega)$ and 
$S_{k-3}(g)\in L^{p_2}(\Omega)$
for each $k\in \mathbb{Z}$, 
we deduce from the estimate \eqref{EQ:Lp2} in Lemma \ref{lem:Lp}, 
H\"older's inequality and the estimate \eqref{EQ:Lp3} for $\alpha=0$ in Lemma \ref{lem:Lp} that
\begin{equation*}
\begin{split}
\Big\| \phi_{j}(\sqrt{\mathH})\Big( f_kS_{k-3}(g)\Big) 
\Big\|_{L^p(\Omega)}
&\le
C
\big\| f_kS_{k-3}(g) \big\|_{L^p(\Omega)}\\
&\le
C
\|f_k\|_{L^{p_1}(\Omega)}
\|S_{k-3}(g)\|_{L^{p_2}(\Omega)}\\
&\le
C
\|f_k\|_{L^{p_1}(\Omega)}
\|g\|_{L^{p_2}(\Omega)},
\end{split}
\end{equation*}
since $1/p=1/p_1+1/p_2$. 
Thus we conclude from the above estimate and Minkowski's inequality that
\begin{equation*}
\begin{split}
I
&\le
C\Biggl\{\sum_{j \in \mathbb{Z}}\bigg(2^{sj}\sum_{|k-j| \le 2}\| f_k\|_{L^{p_1}(\Omega)}\bigg)^q  \Biggr\}^{\frac{1}{q}} \| g \|_{L^{p_2}(\Omega)}\\
&=
C\Biggl\{\sum_{j \in \mathbb{Z}}\bigg(\sum_{|k'| \le 2}2^{-sk'} \cdot 2^{s(j+k')}\|f_{j+k'}\|_{L^{p_1}(\Omega)}\bigg)^q  \Biggr\}^{\frac{1}{q}}
\| g \|_{L^{p_2}(\Omega)}\\
&\le C\sum_{|k'| \le 2}
2^{-sk'}
\bigg\{
\sum_{j \in \mathbb{Z}}
\Big(
2^{s(j+k')}\| f_{j+k'} \|_{L^{p_1}(\Omega)}\Big)^q  \bigg\}^{\frac{1}{q}}
\| g \|_{L^{p_2}(\Omega)}\\
&\le C\|f\|_{\dot{B}^s_{p_1,q}(\mathH)}
\| g \|_{L^{p_2}(\Omega)}.
\end{split}
\end{equation*}

As to the term $III$, 
interchanging the role of $f$ and $g$ in the 
above argument, we get
\[
III 
\le C \| f \|_{L^{p_3}(\Omega)} \|g\|_{\dot{B}^s_{p_4,q}(\mathH)},
\]
where $1/p=1/p_3+1/p_4$. \\

As to the term $V$ for $j - 2 \le k$, 
applying the estimate \eqref{EQ:Lp2}, and using H\"older's inequality, 
we estimate 
\begin{equation*}
\begin{split}
&
\Biggl\{
\sum_{j \in \mathbb{Z}}
\bigg(
2^{sj}
\sum_{j-2 \le k}
\Big\| 
\phi_{j}(\sqrt{\mathH})
\Big(
\sum^{k+2}_{l=k-2} f_k g_l
\Big) 
\Big\|_{L^p(\Omega)} 
\bigg)^q
\Biggr\}^{\frac{1}{q}}\\
\le & \, 
C
\Biggl\{
\sum_{j \in \mathbb{Z}}
\bigg(
2^{sj}
\sum_{j-2 \le k}
\|f_k\|_{L^{p_1}(\Omega)}
\Big(
\sum^{k+2}_{l=k-2} 
\|g_l\|_{L^{p_2}(\Omega)} 
\Big)
\bigg)^q
\Biggr\}^{\frac{1}{q}}\\
\le &\, 
C 
\Biggl\{\sum_{j \in \mathbb{Z}}\bigg(2^{sj}\sum_{j-2 \le k} \| f_k \|_{L^{p_1}(\Omega)} 
\bigg)^q \Biggr\}^{\frac{1}{q}}\|g \|_{L^{p_2}(\Omega)}.
\end{split}
\end{equation*} 
Here, by applying 
Minkowski's inequality to the right member in the above inequality, we find that
\begin{equation*}
\begin{split}
\Biggl\{\sum_{j \in \mathbb{Z}}\bigg(2^{sj}\sum_{j-2 \le k} \| f_k \|_{L^{p_1}(\Omega)} 
\bigg)^q \Biggr\}^{\frac{1}{q}}
&=
\Biggl\{\sum_{j \in \mathbb{Z}}\bigg( \sum_{k' \ge -2} 2^{-sk'} \cdot 2^{s(j+k')}\| f_{j + k'} \|_{L^{p_1}(\Omega)}\bigg)^q \Biggr\}^{\frac{1}{q}}\\
&\le C\sum^\infty_{k'=-2}
2^{-sk'}
\bigg\{
\sum_{j \in \mathbb{Z}}
\Big(
2^{s(j+k')}\| f_{j+k'} \|_{L^{p_1}(\Omega)}\Big)^q  \bigg\}^{\frac{1}{q}}\\
&\le C\|f\|_{\dot{B}^s_{p_1,q}(\mathH)},
\end{split}
\end{equation*} 
since $s>0$. 
Hence, combining the above two estimates, we deduce that 
\[
\Biggl\{
\sum_{j \in \mathbb{Z}}
\bigg(
2^{sj}
\sum_{j-2 \le k}
\Big\| 
\phi_{j}(\sqrt{\mathH})
\Big(
\sum_{|k-l| \le 2} f_k g_l
\Big) 
\Big\|_{L^p(\Omega)} 
\bigg)^q
\Biggr\}^{\frac{1}{q}}
\le 
C\|f\|_{\dot{B}^s_{p_1,q}(\mathH)} \| g \|_{L^{p_2}(\Omega)}.
\]
In a similar way, we can proceed the above argument in the case when $j-2\le l$; thus we
conclude that
\begin{align*}
	V \le C \big( \|f\|_{\dot{B}^s_{p_1,q}(\mathH)} \| g \|_{L^{p_2}(\Omega)} + \| f \|_{L^{p_3}(\Omega)} \|g\|_{\dot{B}^s_{p_4,q}(\mathH)} \big).
\end{align*}

\vspace{0.5cm}

{\bf Case B: Estimates for $II, IV$ and $VI$.}  
First let us estimate the term $II$. 
When $k-j>2$, we deduce from the same argument 
as in $I$ that 
\[
\Biggl\{
\sum_{j \in \mathbb{Z}}
\bigg(2^{sj}\sum_{k-j>2}
\Big\| \phi_{j}(\sqrt{\mathH})\Big( f_kS_{k-3}(g)\Big) 
\Big\|_{L^p(\Omega)}\bigg)^q 
\Biggr\}^{\frac{1}{q}}
\le 
C\|f\|_{\dot{B}^s_{p_1,q}(\mathH)} \| g \|_{L^{p_2}(\Omega)}.
\]
Hence all we have to do is to 
prove the case when $k-j<-2$, i.e., 
\begin{equation}\label{EQ:II}
\Biggl\{
\sum_{j \in \mathbb{Z}}
\bigg(2^{sj}\sum_{k-j<-2}
\Big\| \phi_{j}(\sqrt{\mathH})\Big( f_kS_{k-3}(g)\Big) 
\Big\|_{L^p(\Omega)}\bigg)^q 
\Biggr\}^{\frac{1}{q}}
\le 
C\|f\|_{\dot{B}^s_{p_1,q}(\mathH)} \| g \|_{L^{p_2}(\Omega)}.
\end{equation}
In fact, 
noting from Lemma \ref{lem:decomposition1} that 
\[
f_k, S_{k-3}(g)\in L^\infty(\Omega),
\]
and from \eqref{EQ:inc2} that 
\[
L^\infty(\Omega)\hookrightarrow \mathcal X'(\Omega),
\]
we have 
$$f_k S_{k-3}(g)\in \mathcal X'(\Omega).$$ 
Then we write 
\begin{equation}
\label{EQ:key}
\phi_j(\sqrt{\mathH})\big(f_k S_{k-3}(g)\big) 
= \mathH^{-1}\phi_j(\sqrt{\mathH})\mathH\big(f_k S_{k-3}(g)\big)
\quad \text{in }\mathcal X'(\Omega).
\end{equation}
Here it should be noted that the operator $\mathH^{-1}$ in \eqref{EQ:key} 
is well-defined, since
\[
\mathH^{-1}\phi_j(\sqrt{\mathH}) h \in \mathcal X^\prime(\Omega)
\]
for any $h\in \mathcal X^\prime(\Omega)$.
Hence, 
applying the Leibniz rule in 
Lemma \ref{lem:Leib} to the identities \eqref{EQ:key}, we have:
\begin{equation}\label{EQ:key2}
\begin{split}
&\phi_j(\sqrt{\mathH})\big(f_k S_{k-3}(g)\big)\\
=&\,
\mathH^{-1} \phi_j(\sqrt{\mathH})
\Big\{
(\mathH f_k)S_{k-3}(g)
-
2\nabla f_k \cdot \nabla S_{k-3}(g)
+
f_k \big(\mathH S_{k-3}(g)\big)
\Big\}
\end{split} 
\end{equation} 
in $\mathcal X'(\Omega)$.
Thanks to estimates \eqref{EQ:Lp2} and  \eqref{EQ:Lp3} from Lemma \ref{lem:Lp}, 
the first term in the right member in \eqref{EQ:key2} is estimated as 
\begin{equation*}
\begin{split}
\Big\|\mathH^{-1} \phi_j(\sqrt{\mathH})
\Big\{
(\mathH f_k)S_{k-3}(g)
\Big\}
\Big\|_{L^p(\Omega)}
&\le C 2^{-2j}
\big\|(\mathH f_k)S_{k-3}(g)
\big\|_{L^p(\Omega)}\\
&\le
C 2^{-2j} \|
\mathH f_k\|_{L^{p_1}(\Omega)}
\|S_{k-3}(g)\|_{L^{p_2}(\Omega)}\\
&\le C 
2^{-2(j-k)}
\|f_k\|_{L^{p_1}(\Omega)} 
\|g\|_{L^{p_2}(\Omega)}.
\end{split}
\end{equation*}
In a similar way, we estimate the third term as  
\[
\Big\|
\mathH^{-1} \phi_j(\sqrt{\mathH})
\Big\{
f_k \mathH S_{k-3}(g)
\Big\}
\Big\|_{L^p(\Omega)}
\le C 
2^{-2(j-k)}
\|f_k\|_{L^{p_1}(\Omega)} 
\|g\|_{L^{p_2}(\Omega)}.
\]
As to the second, thanks to \eqref{EQ:nabla2} and \eqref{EQ:nabla3} 
from Lemma \ref{lem:nabla}, we estimate 
\begin{equation*}
\begin{split}
\Big\|
\mathH^{-1} \phi_j(\sqrt{\mathH})
\Big\{
\nabla f_k\cdot \nabla 
S_{k-3}(g)
\Big\}
\Big\|_{L^p(\Omega)}
&\le
C 2^{-2j}
\big\|\nabla f_k\cdot \nabla S_{k-3}(g)
\big\|_{L^p(\Omega)}\\
&\le
C 2^{-2j}
\|\nabla f_k\|_{L^{p_1}(\Omega)}
\|
\nabla S_{k-3}(g)
\|_{L^{p_2}(\Omega)}\\
&\le
C 
2^{-2(j-k)}
\|f_k\|_{L^{p_1}(\Omega)} 
\|g\|_{L^{p_2}(\Omega)}.
\end{split}
\end{equation*}
Hence, combining the identity \eqref{EQ:key2} with the above three estimates, 
we get
\begin{equation*}
\begin{split}
\Big\| 
\phi_{j}(\sqrt{\mathH})\Big( 
f_kS_{k-3}(g)
\Big) \Big\|_{L^p(\Omega)}
\le 
C 
2^{-2(j-k)}
\|f_k\|_{L^{p_1}(\Omega)} 
\|g\|_{L^{p_2}(\Omega)}
\end{split}
\end{equation*}
for any $j, k\in \mathbb Z$. 
Therefore, we conclude from this estimate that
\begin{equation*}
\begin{split}
&\Biggl\{
\sum_{j \in \mathbb{Z}}
\bigg(2^{sj}\sum_{k-j < -2}
\Big\| \phi_{j}(\sqrt{\mathH})\Big( f_kS_{k-3}(g)\Big) 
\Big\|_{L^p(\Omega)}\bigg)^q 
\Biggr\}^{\frac{1}{q}}\\
\le 
&\, 
C \Biggl\{
\sum_{j \in \mathbb{Z}}\bigg( 2^{sj}\sum_{k-j < -2} 2^{-2(j-k)} \|f_k\|_{L^{p_1}(\Omega)} \bigg)^q
\Biggr\}^{\frac{1}{q}} \|g\|_{L^{p_2}(\Omega)}\\
=
&\, 
C \Biggl\{
\sum_{j \in \mathbb{Z}}\bigg( \sum_{k' < -2} 2^{(2-s)k'} \cdot 2^{s(j+k')} \|f_{j+k'}\|_{L^{p_1}(\Omega)}  \bigg)^q
\Biggr\}^{\frac{1}{q}} \|g\|_{L^{p_2}(\Omega)}\\
\le
&\, 
C\|f\|_{\dot{B}^s_{p_1,q}(\mathH)} \| g \|_{L^{p_2}(\Omega)},
\end{split}
\end{equation*}
since $s<2$, 
which proves \eqref{EQ:II}. 
Thus we conclude that
\begin{equation*}
II 
\le C \|f\|_{\dot{B}^s_{p_1,q}(\mathH)}  \| g \|_{L^{p_2}(\Omega)}.
\end{equation*}

As to the term $IV$, interchanging the role of $f$ and $g$ in the above argument, we get
\[
IV \le C \| f \|_{L^{p_3}(\Omega)} \|g\|_{\dot{B}^s_{p_4,q}(\mathH)} .
\]
As to the term $VI$, 
we estimate in a similar way to $II$;
\begin{equation*}
\begin{split}
VI
& 
\le \Biggl\{
\sum_{j \in \mathbb{Z}}
\bigg(
2^{sj}
\sum_{j-2 > k}
\Big\| 
\phi_{j}(\sqrt{\mathH})
\Big(
\sum_{|k-l| \le 2} f_k g_l
\Big) 
\Big\|_{L^p(\Omega)} 
\bigg)^q
\Biggr\}^{\frac{1}{q}}\\
&
\le
C \Biggl\{
\sum_{j \in \mathbb{Z}} \Big(2^{sj}\sum_{j-2 > k} 2^{-2(j-k)}\|f_k\|_{L^{p_1}(\Omega)} \Big)^q 
\Biggr\}^{\frac{1}{q}} \| g \|_{L^{p_2}(\Omega)}\\
&
=
C  \Biggl\{
\sum_{j \in \mathbb{Z}}
\Big(
\sum_{k' < -2} 2^{(2-s)k'} \cdot 2^{s(j+k')}\| f_{j+k'} \|_{L^{p_1}(\Omega)} \Big)^q 
\Biggr\}^{\frac{1}{q}} \| g \|_{L^{p_2}(\Omega)}\\
& 
\le 
C\|f\|_{\dot{B}^s_{p_1,q}(\mathH)} \| g \|_{L^{p_2}(\Omega)},
\end{split}
\end{equation*}
since $s<2$.\\

Summarizing cases A and B, we arrive at the required estimate \eqref{EQ:bilinear2}. 
The proof of the case when $1\le p_2,p_3<\infty$ is finished.\\

It remains to prove the case when $p_2=\infty$ or $p_3=\infty$.\\ 

{\bf The case: $p_2=\infty$ or $p_3=\infty$.} 
We may prove only 
the case when $p_2=p_3=\infty$, since the other cases are 
proved in a similar way. In this case, we note that
\[
p_1=p_4=p.
\]
Let $f,g\in \dot{B}^s_{p, q}(\mathH)\cap 
L^{\infty}(\Omega)$. Then it follows from Lemma \ref{lem:Lp} that
\begin{equation}\label{EQ:Prep}
\Big\|\sum_{j=k}^\infty
f_j
\Big\|_{L^\infty(\Omega)}
\le C\|f\|_{L^\infty(\Omega)}
\end{equation}
for any $k\in\mathbb Z$. Hence there exist 
a subsequence $$\Big\{\sum_{j=k_l}^\infty
f_j\Big\}_{l\in\mathbb N}$$
and a function
$F\in L^\infty(\Omega)$ such that 
\begin{equation}\label{EQ:11-convergence}
\sum_{j=k_l}^\infty
f_j \rightharpoonup 
F\quad \text{weakly* in }L^\infty(\Omega)
\end{equation}
as $l\to \infty$, 
which also yields the convergence in $\mathcal X'(\Omega)$ 
and $\mathcal Z' (\Omega)$ 
by the embedding 
$$L^\infty (\Omega) \hookrightarrow \mathcal X' (\Omega) 
\hookrightarrow \mathcal Z' (\Omega).
$$ 
On the other hand, it follows from Lemma \ref{lem:decomposition1} that
\[
\sum_{j=k_l}^\infty
f_j \to
f\quad \text{in }\mathcal Z'(\Omega)
\]
as $l\to \infty$. Hence we see that $F=f$ in $\mathcal Z'(\Omega)$, 
which implies that
\[
P_f:= f-F \in \mathcal P(\Omega).
\]
Therefore we conclude from 
\eqref{EQ:11-convergence} that 
\begin{equation}\label{EQ:1111}
\sum_{j=k_l}^\infty
f_j \rightharpoonup 
f-P_f\quad \text{weakly* in }L^\infty(\Omega)
\end{equation}
as $l\to \infty$. 
In a similar way, there exist 
a subsequence $$\Big\{\sum_{j=k_{l'}}^\infty
g_j\Big\}_{l'\in\mathbb N}$$
and $P_g\in \mathcal P(\Omega)$ such that 
\begin{equation}\label{EQ:2222}
\sum_{j=k_{l'}}^\infty
g_j \rightharpoonup 
g-P_g\quad \text{weakly* in }L^\infty(\Omega)
\end{equation}
as $l'\to \infty$. 
Hence, by \eqref{EQ:1111} and \eqref{EQ:2222}, 
there exists a subsequence $\{ l'(l) \}_{l = 1}^\infty$ 
of $\{ l' \}_{l'=1}^\infty$ such that
\[
\Big(\sum_{j=k_l}^\infty
f_j\Big)
\Big(\sum_{j=k_{l'(l)}}^\infty
g_j \Big)
\rightharpoonup
(f-P_f)(g-P_g)
\quad \text{weakly* in }L^\infty(\Omega)
\]
as $l\to \infty$. 
Hence we have
\begin{equation}\label{EQ:condi-1}
\Big(\sum_{j=k_l}^\infty
f_j\Big)
\Big(\sum_{j=k_{l'(l)}}^\infty
g_j \Big)
\to
(f-P_f)(g-P_g)
\quad \text{in }\mathcal X'(\Omega)
\end{equation}
as $l\to \infty$, since $L^\infty(\Omega)\hookrightarrow \mathcal X^\prime(\Omega)$. 
Now, 
 the estimate of 
$\dot{B}^s_{p,q}$-norm of the left member in \eqref{EQ:condi-1} is obtained by 
the argument as in the previous case 
$1 \le p_2,p_3<\infty$. Hence, there exists a 
constant $C>0$ such that
\begin{equation}\label{EQ:condi-2}
\bigg\|\Big(\sum_{j=k_l}^\infty
f_j\Big)
\Big(\sum_{j=k_{l'(l)}}^\infty
g_j \Big)\bigg\|_{\dot{B}^s_{p,q}(\mathH)}
\le 
C\left(
\|f\|_{\dot{B}^s_{p,q}(\mathH)}
\|g\|_{L^\infty(\Omega)}
+
\|f\|_{L^\infty(\Omega)}
\|g\|_{\dot{B}^s_{p,q}(\mathH)}
\right)
\end{equation}
for any $l\in\mathbb N$. 
Here, we note that $P_f$ and $P_g$ are 
constants by the assertion (iii) from  
Lemma \ref{lem:P}.
As a consequence of \eqref{EQ:condi-1} and 
\eqref{EQ:condi-2}, we conclude from Lemma \ref{lem:Fatou} that
\begin{equation*}
\begin{split}
\|fg\|_{\dot{B}^s_{p,q}(\mathH)}
\le & \, \liminf_{l\to\infty}
\bigg\|\Big(\sum_{j=k_l}^\infty
f_j\Big)
\Big(\sum_{j=k_{l'(l)}}^\infty
g_j \Big)\bigg\|_{\dot{B}^s_{p,q}(\mathH)}\\
&+\|fP_g\|_{\dot{B}^s_{p,q}(\mathH)}
+\|P_fg\|_{\dot{B}^s_{p,q}(\mathH)}
+\|P_fP_g\|_{\dot{B}^s_{p,q}(\mathH)}\\
\le & \, C\left(
\|f\|_{\dot{B}^s_{p,q}(\mathH)}
\|g\|_{L^\infty(\Omega)}
+
\|f\|_{L^\infty(\Omega)}
\|g\|_{\dot{B}^s_{p,q}(\mathH)}
\right)\\
&+\|f\|_{\dot{B}^s_{p,q}(\mathH)}|P_g|
+|P_f|\|g\|_{\dot{B}^s_{p,q}(\mathH)}
+\|P_fP_g\|_{\dot{B}^s_{p,q}(\mathH)}.
\end{split}
\end{equation*}
Here, combining 
part (c) in (i) and the assertion (iii) from  
Lemma \ref{lem:P}, we deduce that
\[
\|P_fP_g\|_{\dot{B}^s_{p,q}(\mathH)}=0. 
\]
Hence, all we have to do is to prove that
\begin{equation}\label{EQ:111}
|P_f| \le C
\|f\|_{L^\infty(\Omega)},
\end{equation}
\begin{equation}\label{EQ:112}
|P_g|\le C
\|g\|_{L^\infty(\Omega)}.
\end{equation}
Noting \eqref{EQ:1111}, we estimate, by using \eqref{EQ:Prep}, 
\[
|P_f|\le 
\|f\|_{L^\infty(\Omega)}
+ 
\liminf_{l\to\infty}
\Big\|
\sum_{j=k_l}^\infty
f_j
\Big
\|_{L^\infty(\Omega)}
\le C\|f\|_{L^\infty(\Omega)}.
\]
This proves \eqref{EQ:111}. In a 
similar way, we get \eqref{EQ:112}.
The proof of Theorem \ref{thm:bilinear} is finished.

\section{A final remark} \label{sec:sec5}
Once the bilinear estimates for the Dirichlet Laplacian are established, 
the same type estimates in Besov spaces generated by the Schr\"odinger operators 
are also obtained. In this section we discuss this topic. \\

To begin with, let us give definitions of function spaces generated by 
the Schr\"odinger operators along \cite{IMT-Besov}. 
Let $\Omega$ be an open set in $\mathbb R^n$ 
with $n \ge 1$. 
We denote by $\mathH_V$  
the self-adjoint realization of $-\Delta + V$ with the domain
\[
\mathcal D(\mathH_V)
=
\big\{ f \in H^1_0(\Omega)  \, \big| \, 
\mathH_V f \in L^2 (\Omega), \, \sqrt{V_+} f\in L^2(\Omega)\big\}
\]
such that 
\begin{equation*} 
\left(\mathcal{H}_Vf,g\right)_{L^2(\Omega)}
=\int_\Omega \nabla f(x)\cdot \overline{\nabla g(x)}\, dx
+\int_\Omega V(x)f(x)\overline{g(x)}\, dx
\end{equation*}
for any $f\in \mathcal D(\mathH_V)$ and 
$g\in H^1_0(\Omega)$ with $\sqrt{V_+}g\in L^2(\Omega)$, 
where $V=V(x)$ is a real-valued measurable function on $\Omega$ such that 
\begin{equation}\label{EQ:assV1}
V = V_{+} - V_-, \quad 
V_{\pm} \geq 0, \quad 
V_+ \in  L^1_{\rm loc} (\Omega) 
\text{ and } 
V_- \in K_n (\Omega). 
\end{equation}
Here, we say that $V_-\in K_{n}( \Omega)$ if 
\begin{align}\notag 
\left\{
\begin{aligned}
	&\lim_{r \rightarrow 0} \sup_{x \in \Omega} \int_{\Omega \cap \{|x-y|<r\}} 
	   \frac{V_-(y)}{|x-y|^{n-2}} \,dy = 0 &\text{for }n\ge 3, \\
	&\lim_{r \rightarrow 0} \sup_{x \in \Omega} \int_{\Omega \cap \{|x-y|<r\}} 
	   \log (|x-y|^{-1})V_-(y) \,dy = 0 &\text{for }n=2, \\
	&\sup_{x \in \Omega}\int_{\Omega \cap \{|x-y|<1\}} V_-(y) \,dy <\infty &\text{for }n=1. 
	\end{aligned}\right.
\end{align}
As to the negative part $V_-$ of $V$, we proved in Lemma 2.3 from \cite{IMT-bdd} that $\sqrt{V_-}f \in L^2(\Omega)$, provided that $V_-\in K_n(\Omega)$ and $f \in H^1_0(\Omega)$.
Then we define a linear topological space $\mathcal X_V(\Omega)$, 
its dual space $\mathcal X'_V(\Omega)$ 
and inhomogeneous Besov spaces $B^s_{p,q}(\mathH_V)$ 
in a similar way to definitions in 
\S \ref{sec:sec2}. 
Furthermore, 
if we assume the additional condition that 
\begin{gather}
\label{EQ:assV2}
\begin{cases} 
\displaystyle \sup _{x \in \Omega} 
   \int_{\Omega} \dfrac{V_- (y)}{|x-y|^{n-2}} \, dy 
 < \dfrac{\pi^{\frac{n}{2}}}{\Gamma \big(\frac{n}{2} -1\big)}
& \quad \text{if } n \geq 3,\\
V _- = 0
& \quad \text{if } n = 1,2, 
\end{cases}
\end{gather}
then we also define a 
linear topological space $\mathcal Z_V(\Omega)$, 
its dual space $\mathcal Z'_V(\Omega)$ 
and homogeneous Besov spaces $\dot{B}^s_{p,q}(\mathH_V)$ 
in a similar way to definitions in 
\S\ref{sec:sec2}.\\

We have proved the following result in \cite{IMT-Besov}.
\begin{prop}[Proposition 3.5 in \cite{IMT-Besov}]
\label{prop:equiv}
Let $\Omega$ be an open set of $\mathbb R^n$, and let $1 \le p,q\le \infty$ and 
$s$ be such that 
\begin{equation*}
\label{EQ:s}
\begin{cases}
\displaystyle
-\min
\left\{
2, n\left(1-\frac{1}{p}\right)
\right\} 
< s <
\min\left\{\frac{n}{p},2\right\}
\quad&\text{if }n\ge3,\\
\displaystyle
-2+\frac{2}{p}
<s<
\frac{2}{p}
&\text{if }n=1,2.
\end{cases}
\end{equation*}
Then the following assertions hold{\rm :} 
\begin{itemize}
\item[(i)] 
Suppose that the potential $V$ satisfies the assumption \eqref{EQ:assV1} and 
\begin{gather*}\label{EQ;V-I}
\begin{cases} 
V \in L^{\frac{n}{2},\infty} (\Omega) + L^\infty(\Omega) 
& \quad \text{if } n \geq 3,\\
V \in K_n(\Omega) 
& \quad \text{if } n = 1,2,
\end{cases}
\end{gather*}
where $L^{\frac{n}{2},\infty}(\Omega)$ is the Lorentz space. 
Then 
\begin{equation*}
\label{EQ:equi1}
B^s_{p,q}(\mathH_V)  \cong B^s_{p,q}(\mathH).
\end{equation*}
\item[(ii)] 
Let $n\ge2$. 
Suppose that the potential $V$ satisfies the assumption 
\eqref{EQ:assV2}  and 
\begin{gather*}\label{EQ;V-H}
\begin{cases}
V \in L^{\frac{n}{2},\infty} (\Omega)
& \quad \text{if } n \geq 3,\\
V \in L^1(\Omega)
& \quad \text{if } n = 2.
\end{cases}
\end{gather*}
Then 
\begin{equation*}
\label{EQ:equi2}
\dot{B}^s_{p,q}(\mathH_V)  \cong \dot{B}^s_{p,q}(\mathH).
\end{equation*}
\end{itemize}
\end{prop}

As an immediate consequence of Theorem \ref{thm:bilinear} and Proposition \ref{prop:equiv}, 
we have the following.

\begin{thm}
\label{thm:bilinearV}
Let $p,p_1,p_2,p_3,p_4$ and $q$ be such that
\[
1 \le p, p_1,p_2,p_3,p_4,q\le \infty 
\quad \text{and}\quad 
\frac{1}{p}=\frac{1}{p_1}+\frac{1}{p_2}=\frac{1}{p_3}+\frac{1}{p_4},
\]
and let $s$ be such that
\[ 
0<s <\min\left\{\frac{n}{p_1},\frac{n}{p_4},2\right\}\quad \text{if }n\ge3;
\quad 
0<s< \min\left\{\frac{2}{p_1},\frac{2}{p_4}\right\} \quad\text{if }n=1,2.
\]
Then, 
under the same assumption on $V$ 
 in Proposition \ref{prop:equiv}, 
the assertions {\rm (i)} and {\rm (ii)} in Theorem \ref{thm:bilinear} hold for $B^s_{p,q}(\mathH_V)$ and $\dot{B}^s_{p,q}(\mathH_V)$, respectively. 
\end{thm}


\appendix
 \section{(High regularity case)
} 
 \label{App:AppendixA}
In this  appendix 
we check that the bilinear estimates do not necessarily hold for some $s\ge2$. 
Let us consider the bilinear estimate 
\eqref{EQ:bilinear2} in the case when 
\[
p=1,\quad p_1=p_2=p_3=p_4=q=2\quad\text{and}\quad f=g,
\]
namely, 
\begin{equation}
\label{EQ:counter1}
\|f^2\|_{\dot{B}^s_{1,2}(\mathH)}
\le C
\|f\|_{\dot{B}^s_{2,2}(\mathH)} \|f\|_{L^2(\Omega)}
\end{equation}
for any 
$f\in \dot{B}^s_{2,2}(\mathH)\cap L^2(\Omega)$. 
We note that 
the estimate \eqref{EQ:counter1} is already proved 
for any $0<s<2$ on an arbitrary open set 
(see the case (ii) in \S\ref{sec:sec2}). 
We shall show that the estimate \eqref{EQ:counter1} 
does not hold for some $s \ge 2$.\\

Let $n \ge 3$ and $\Omega$ be an exterior domain in $\mathbb R^n$ such that $\R^n\setminus \Omega$ 
is compact and its boundary is of $C^{1,1}$. 
Then 
it is known that 
\begin{equation}
\label{EQ:counter3}
\|\nabla e^{-t\mathH}\|_{L^1(\Omega)\to L^\infty(\Omega)}
\ge C t^{-\frac{n}{2}}
\end{equation}
for sufficiently large $t>0$ 
(see Ishige and Kabeya \cite{IshKab-2007}, 
and also Zhang \cite{Zhang-2003}). 
However 
we can claim the following:\\

\noindent{\bf Claim A.1.} 
Let $\varepsilon>0$. 
If the estimate \eqref{EQ:counter1} holds 
for any $s\in[2,n+2+\varepsilon]$, then 
there exists a constant $C>0$ such that
\begin{equation}
\label{EQ:counter2}
\|\nabla e^{-t\mathH}\|_{L^1(\Omega)\to L^\infty(\Omega)}
\le C 
t^{-\frac{n}{2}-\frac{1}{2}+\frac{\varepsilon}{4}}
\end{equation}
for sufficiently large $t>0$. \\

The estimate \eqref{EQ:counter2} 
contradicts \eqref{EQ:counter3} if we choose 
$\varepsilon$ sufficiently small.
Thus, if Claim A.1 is proved, then 
we conclude that when $\Omega$ is the exterior domain whose boundary is compact and of 
$C^{1,1}$,
the bilinear estimate \eqref{EQ:bilinear2} does not hold for some $s\ge2$. \\

From now on, we prove Claim A.1. Let $f\in L^1(\Omega)$. 
By the Leibniz rule, we have 
\[
\mathH (e^{-t\mathH} f)^2
=
2 (\mathH e^{-t\mathH} f)(e^{-t\mathH} f)
-2|\nabla e^{-t\mathH} f|^2 
\quad \text{in }\mathscr D'(\Omega), 
\]
and hence,
\begin{equation}
\label{EQ:B-1}
\begin{split}
\|\nabla e^{-t\mathH} f\|^2_{L^\infty(\Omega)}
&\le \|\mathH (e^{-t\mathH} f)^2
\|_{L^\infty(\Omega)}
+
\|(\mathH e^{-t\mathH} f)(e^{-t\mathH} f)\|_{L^\infty(\Omega)}
\\
&=:
I + II.
\end{split}
\end{equation}
We readily see from Proposition \ref{prop:Lp} for $p=1$ and $q=\infty$ that 
\begin{equation}
\label{EQ:B-2}
\begin{split}
II
& \le 
\|\mathH e^{-t\mathH} f\|_{L^\infty(\Omega)}
\|e^{-t\mathH} f\|_{L^\infty(\Omega)}\\
& \le  
Ct^{-\frac{n}{2}-1}\|f\|_{L^1(\Omega)}\cdot t^{-\frac{n}{2}}\|f\|_{L^1(\Omega)}\\
&= 
Ct^{-n-1}\|f\|^2_{L^1(\Omega)}.
\end{split}
\end{equation}
As to the estimate for $I$, we recall that 
\begin{equation}\label{EQ:PHI}
\phi_j=\Phi_j\phi_j,
\end{equation}
where
\[
\Phi_j=\phi_{j-1}+\phi_j+\phi_{j+1}.
\]
Then, by using identities \eqref{EQ:PHI} and the part (ii) from 
Lemma \ref{lem:Lp} for $p=\infty$ and $\alpha=1$, 
we find that 
\begin{equation*}
\begin{split}
I
& \le  
\sum_{j\in\mathbb Z}\|\phi_j(\sqrt{\mathH})\mathH (e^{-t\mathH} f)^2\|_{L^\infty(\Omega)}\\
& =  
\sum_{j\in\mathbb Z}\|\mathH \Phi_j(\sqrt{\mathH})  \phi_j(\sqrt{\mathH}) (e^{-t\mathH} f)^2\|_{L^\infty(\Omega)}\\
& \le  
C\sum_{j\in\mathbb Z}2^{2j}\| \phi_j(\sqrt{\mathH}) (e^{-t\mathH} f)^2\|_{L^\infty(\Omega)}.
\end{split}
\end{equation*}
Here, by using \eqref{EQ:PHI} and 
Proposition \ref{prop:Lp} for $p=1$, $q=\infty$ and $\theta=2^{-2j}$, we estimate 
\begin{equation*}
\begin{split}
\| \phi_j(\sqrt{\mathH}) (e^{-t\mathH} f)^2\|_{L^\infty(\Omega)}
& =  
\| \Phi_j(\sqrt{\mathH})  \phi_j(\sqrt{\mathH}) (e^{-t\mathH} f)^2\|_{L^\infty(\Omega)}\\
& \le 
C2^{nj}\|\phi_j(\sqrt{\mathH})(e^{-t\mathH} f)^2\|_{L^1(\Omega)}.
\end{split}
\end{equation*}
Hence, combining these estimates obtained now, 
we get
\begin{equation*}
\begin{split}
I
& \le
C\sum_{j\in\mathbb Z}2^{(n+2)j}\|\phi_j(\sqrt{\mathH})(e^{-t\mathH} f)^2\|_{L^1(\Omega)}\\
& =: C(I_1+I_2), 
\end{split}
\end{equation*}
where 
\[
I_1=\sum_{j\le 0}2^{(n+2)j}\|\phi_j(\sqrt{\mathH})(e^{-t\mathH} f)^2\|_{L^1(\Omega)},
\]
\[
I_2=\sum_{j\ge 1}2^{(n+2)j}\|\phi_j(\sqrt{\mathH})(e^{-t\mathH} f)^2\|_{L^1(\Omega)}.
\]
Here, writing
\[
I_1=\sum_{j\le 0}2^{\varepsilon j}\cdot2^{-\varepsilon j}\cdot2^{(n+2)j}\|\phi_j(\sqrt{\mathH})(e^{-t\mathH} f)^2\|_{L^1(\Omega)},
\]
\[
I_2=\sum_{j\ge 1}2^{-\varepsilon j}\cdot
2^{\varepsilon j}\cdot2^{(n+2)j}\|\phi_j(\sqrt{\mathH})(e^{-t\mathH} f)^2\|_{L^1(\Omega)}
\]
for any $\ep>0$, 
we estimate
\begin{equation*}
\begin{split}
I_1 & \le
\Big\{\sum_{j\le 0}2^{2\varepsilon j}
\Big\}^{\frac12}
\Big\{\sum_{j\le 0}\big(2^{(n+2-\varepsilon)j}\|\phi_j(\sqrt{\mathH})(e^{-t\mathH} f)^2
\|_{L^1(\Omega)}\big)^2\Big\}^{\frac12}\\
& \le
C\|(e^{-t\mathH} f)^2
\|_{\dot{B}^{n+2-\varepsilon}_{1,2}(\mathH)},
\end{split}
\end{equation*}
\begin{equation*}
\begin{split}
I_2 &\le 
\Big\{\sum_{j\ge 1}2^{-2\varepsilon j}\Big\}^{\frac{1}{2}}
\Big\{\sum_{j\ge 1}\big(2^{(n+2+\varepsilon)j}\|\phi_j(\sqrt{\mathH})(e^{-t\mathH} f)^2\|_{L^1(\Omega)}\big)^2\Big\}^{\frac{1}{2}}\\
& \le 
C\|(e^{-t\mathH} f)^2\|_{\dot{B}^{n+2+\varepsilon}_{1,2}(\mathH)},
\end{split}
\end{equation*}
respectively, which imply that 
\begin{equation}\label{EQ:THE1}
\begin{split}
I \le 
C\left\{\|(e^{-t\mathH} f)^2\|_{\dot{B}^{n+2-\varepsilon}_{1,2}(\mathH)}
+
\|(e^{-t\mathH} f)^2\|_{\dot{B}^{n+2+\varepsilon}_{1,2}(\mathH)}
\right\}
\end{split}
\end{equation}
for any $\varepsilon > 0$. 
Now, since $f\in L^1(\Omega)$, it follows from
$L^1$-$L^2$-estimate for heat semigroup $e^{-t\mathH}$ that 
\[
e^{-t\mathH} f \in \dot{B}^s_{2,2}(\mathH)\cap L^2(\Omega) 
\quad \text{for any $s\ge0$ and 
$t>0$.}
\] 
Hence, applying the assumption that \eqref{EQ:counter1} holds for any 
$s\in[2,n+2+\varepsilon]$, we deduce that
\begin{equation}\label{EQ:exp1}
\|(e^{-t\mathH} f)^2\|_{\dot{B}^{n+2-\varepsilon}_{1,2}(\mathH)}
\le C
\|e^{-t\mathH} f\|_{\dot{B}^{n+2-\varepsilon}_{2,2}(\mathH)}
\|e^{-t\mathH} f\|_{L^2(\Omega)}.
\end{equation}
Since
\[
\|g\|_{\dot{B}^{s}_{2,2}(\mathH)} \simeq \|\mathH^{\frac{s}{2}} g\|_{L^2(\Omega)},\quad g \in \dot{B}^{s}_{2,2}(\mathH)
\]
for any $s\in \R$,
the first factor in the right member of 
\eqref{EQ:exp1} is estimated as 
\begin{equation*}
\begin{split}
\|e^{-t\mathH} f
\|_{\dot{B}^{n+2-\varepsilon}_{2,2}(\mathH)}
&\le C
\|\mathH^{\frac{n}{2}+1-\frac{\varepsilon}{2}} e^{-t\mathH} f\|_{L^2(\Omega)}\\
&\le 
C
t^{-\frac{n}{2}-1+\frac{\varepsilon}{2}}
\|e^{-\frac{t}{2}\mathH} f\|_{L^2(\Omega)}\\
&\le 
C
t^{-\frac{3n}{4}-1+\frac{\varepsilon}{2}}
\|f\|_{L^1(\Omega)},
\end{split}
\end{equation*}
where we used Proposition \ref{prop:Lp} in the second step, and $L^1$-$L^2$-estimate for heat semigroup $e^{-\frac{t}{2}\mathH}$ in the last step. 
Again, by $L^1$-$L^2$-estimate for heat semigroup $e^{-\frac{t}{2}\mathH}$, we have
\[
\|e^{-t\mathH} f\|_{L^2(\Omega)}
\le C t^{-\frac{n}{4}}
\|f\|_{L^1(\Omega)}.
\]
Hence, combining all the estimates obtained now, we get
\begin{equation}\label{EQ:THE2}
\|(e^{-t\mathH} f)^2\|_{\dot{B}^{n+2-\varepsilon}_{1,2}(\mathH)}
\le C
t^{-n-1+\frac{\varepsilon}{2}}
\|f\|^2_{L^1(\Omega)}.
\end{equation}
In a similar way, we have 
\begin{equation}\label{EQ:THE3}
\|(e^{-t\mathH} f)^2\|_{\dot{B}^{n+2+\varepsilon}_{1,2}(\mathH)}
\le C
t^{-n-1-\frac{\varepsilon}{2}}
\|f\|^2_{L^1(\Omega)}.
\end{equation}
Therefore, combining 
the estimates \eqref{EQ:THE1}, \eqref{EQ:THE2} 
and \eqref{EQ:THE3}, we conclude that
\begin{equation}
\label{EQ:B-3}
\begin{split}
I
\le C
\left(
t^{-n-1+\frac{\varepsilon}{2}}
+
t^{-n-1-\frac{\varepsilon}{2}}
\right)
\|f\|^2_{L^1(\Omega)}.
\end{split}
\end{equation}
Thus, combining \eqref{EQ:B-1}, \eqref{EQ:B-2} 
and \eqref{EQ:B-3}, we arrive at \eqref{EQ:counter2}. 
Claim A.1 is proved.


\section{} 
\label{App:AppendixB}
In this appendix we prove the following. 
\begin{lem}\label{lem:norm}
Let $1 \le p \le \infty$ and $T$ be a bounded linear operator from $L^p(\Omega)$ to $L^\infty(\Omega)$, and $T(x,y)$ the kernel of $T$. Then
\begin{equation*} 
\|T\|_{L^p(\Omega)\to L^\infty(\Omega)} 
=
\|T(\cdot,\cdot)\|_{L^\infty(\Omega; L^{p'}(\Omega))},
\end{equation*}
where $p'$ is the conjugate exponent of $p$. 
\end{lem}
\begin{proof}
We have: 
\begin{equation}
\label{EQ:1st}
\|T\|_{L^p(\Omega)\to L^\infty(\Omega)} 
\le \|T(\cdot,\cdot)\|_{L^\infty(\Omega; L^{p'}(\Omega))} 
\end{equation}
for any $1 \le p \le \infty$. 
In fact, let $f \in L^p(\Omega)$. Then it follows from H\"older's inequality that
\begin{equation*}
\begin{split}
|Tf(x)| &=
\Big|\int_\Omega T(x,y)f(y)\,dy\Big|\\
&\le 
\|T(x,\cdot)\|_{L^{p'}(\Omega)} \|f\|_{L^p(\Omega)}
\end{split}
\end{equation*}
for a.e.\,$x\in \Omega$. Hence we obtain
\[
\|Tf\|_{L^\infty(\Omega)}
\le \|T(\cdot,\cdot)\|_{L^\infty(\Omega; L^{p'}(\Omega))} \|f\|_{L^p(\Omega)},
\]
which implies \eqref{EQ:1st}. 
Therefore it suffices to prove the converse:
\begin{equation}
\label{EQ:2nd}
\|T(\cdot,\cdot)\|_{L^\infty(\Omega; L^{p'}(\Omega))} 
\le \|T\|_{L^p(\Omega)\to L^\infty(\Omega)} 
\end{equation}
for any $1 \le p \le \infty$.
When $1 \le p < \infty$, 
we estimate 
\begin{equation*}
\begin{split}
\|T(x,\cdot)\|_{L^{p'}(\Omega)}
&=
\sup_{f\in L^p(\Omega),\,\|f\|_{L^p(\Omega)}=1}
\Big|\int_\Omega T(x,y)f(y)\,dy\Big|\\
&= 
\sup_{f\in L^p(\Omega),\,\|f\|_{L^p(\Omega)}=1}
\big|T f(x)\big|\\
& \le
\sup_{f\in L^p(\Omega),\,\|f\|_{L^p(\Omega)}=1}
\|T\|_{L^p(\Omega)\to L^\infty(\Omega)} \|f\|_{L^p(\Omega)}\\
&\le
\|T\|_{L^p(\Omega)\to L^\infty(\Omega)}
\end{split}
\end{equation*}
for any $x \in \Omega$. 
This proves \eqref{EQ:2nd} for $1\le p<\infty$.
When $p=\infty$, 
fixing $x_0 \in \Omega$, we estimate
\begin{equation*}
\begin{split}
\|T(x_0,\cdot)
\|_{L^1(\Omega)} &= 
\int_{\Omega} |T(x_0,y)|\,dy\\
&=
\int_{\Omega} T(x_0,y)
e^{-i\arg{\{T(x_0,y)\}}}\,dy\\
& \le 
\sup_{x\in\Omega}
\Big|
\int_{\Omega} T(x,y)
e^{-i\arg{\{T(x_0,y)\}}}\,dy\Big|\\
&=
\sup_{x\in\Omega}
\big|T e^{-i\arg{\{T(x_0,\cdot)\}}}(x)\big|\\
&\le
\|T\|_{L^\infty(\Omega)\to L^\infty(\Omega)}
\big\|e^{-i\arg{\{T(x_0,\cdot)\}}} \big\|_{L^\infty(\Omega)}\\
&= \|T\|_{L^\infty(\Omega)\to L^\infty(\Omega)},
\end{split}
\end{equation*}
which proves \eqref{EQ:2nd} for $p=\infty$.
The proof of Lemma \ref{lem:norm} is finished.
\end{proof}


\end{document}